\documentclass[12pt]{amsart}
 \usepackage{amssymb,amsmath,amstext}
\usepackage{amsthm,amsmath,amsfonts,amssymb,amscd,latexsym,stmaryrd}
\usepackage[colorlinks=false, urlcolor=blue, linkcolor=blue,
citecolor=blue]{hyperref} 
\usepackage[margin=1in]{geometry}
\usepackage{color}
\usepackage[usenames,dvipsnames]{xcolor}
\usepackage{colortbl}
\usepackage[all,cmtip]{xy}
\usepackage{graphicx}
\usepackage{ytableau}
\usepackage{arydshln}
\usepackage{blkarray}
\usepackage{arydshln}
\usepackage {tikz}
\usepackage{graphicx}
\usepackage {cite}
\usepackage{comment}
\usetikzlibrary{arrows,decorations.pathmorphing,backgrounds,positioning,fit,matrix}
\definecolor{processblue}{cmyk}{0.96,0,0,0} \usepackage{enumerate}
\input{xy} \xyoption{all} \usepackage{xypic}
\hypersetup{
  colorlinks   = true, %Colours links instead of ugly boxes
  urlcolor     = blue, %Colour for external hyperlinks
  linkcolor    = blue, %Colour of internal links
  citecolor   = blue %Colour of citations
}
\numberwithin{equation}{section}

\theoremstyle{plain} \newtheorem{theorem}{Theorem}[section]
\newtheorem{proposition}[theorem]{Proposition} \newtheorem{lemma}[theorem]{Lemma}
\newtheorem{corollary}[theorem]{Corollary}
\newtheorem{conjecture}[theorem]{Conjecture}

\theoremstyle{definition} \newtheorem{definition}[theorem]{Definition}
\newtheorem{example}[theorem]{Example} \newtheorem{remark}[theorem]{Remark}

\DeclareMathOperator{\Hess}{Hess}\DeclareMathOperator{\rk}{rk}
\DeclareMathOperator{\ann}{Ann}
\DeclareMathOperator{\Cat}{Cat}\DeclareMathOperator{\B}{\mathcal{B}}
\DeclareMathOperator{\K}{\mathsf{k}}
\DeclareMathOperator{\diag}{\mathrm{diag}}

\begin{document}

\author{Nasrin Altafi} \address{Department of Mathematics, KTH Royal
  Institute of Technology, S-100 44 Stockholm, Sweden}
\email{nasrinar@kth.se}

\title[]{Jordan types with small parts for Artinian Gorenstein algebras of codimension three}
\maketitle

\begin{abstract}
We study Jordan types of linear forms for  graded Artinian  Gorenstein algebras having arbitrary codimension. We introduce rank matrices of linear forms for such algebras that represent the ranks of multiplication maps in various degrees. We show that there is a 1-1 correspondence between rank matrices and Jordan degree types. For Artinian Gorenstein algebras with codimension three we classify all rank matrices that occur for linear forms with vanishing third power. As a consequence,  we show for such algebras that the possible Jordan types with parts of length at most four are uniquely determined by at most three parameters.
\end{abstract}

\section{Introduction}
The Jordan type of a graded Artinian algebra $A$ and linear form $\ell$ is a partition determining the Jordan block decomposition for the (nilpotent) multiplication map by $\ell$ on $A$ which is denoted by $P_{\ell,A}=P_\ell$. Jordan type determines the weak and strong Lefschetz properties of Artinian  algebras. A graded Artinian algebra $A$ is said to satisfy the weak Lefschetz property (WLP) if multiplication map by a linear form on $A$ has maximal rank in every degree. If this holds for all powers of a linear form the algebra $A$ is said to have the strong Lefschetz property (SLP). It is known that an Artinian  algebra $A$ has the WLP if there is a linear form $\ell$ where the number of parts in $P_\ell$ is equal to the Sperner number of $A$, the maximum value of the Hilbert function $h_A$. Also $A$ has the SLP if there is a linear form $\ell$ such that $P_\ell = h^\vee_A$ the conjugate partition of $h_A$ see \cite{IMM}.
Jordan type of a linear form for an Artinian algebra captures more information than the weak and Strong Lefschetz properties. Recently, there has been studies about Jordan types of Artinian  algebras also in more general settings, see \cite{IMM, IMM2, IKVZ} and their references.  
Studying Artinian Gorenstein algebras is of great interest among the researchers in the area.   Gorenstein algebras are commutative Poincar\'e duality algebras \cite{MW} and thus natural algebraic objects to cohomology rings of smooth complex projective varieties. 
There has been many studies in the Lefschetz properties and Jordan types of Artinian  Gorenstein algebras \cite{CG, HW, GZ, Gondim, MH}.
  Gorenstein algebras of codimension two are  complete intersections and they all satisfy the SLP.  The list of all possible Jordan types of linear forms, not necessarily generic linear forms, for complete intersection algebras of codimension two is provided in \cite{CIJT}. 
   
In this article, we  study the ranks of multiplication maps by linear forms on graded Artinian  Gorenstein algebras that are quotients of polynomial ring $S = \K[x_1,\dots ,x_n]$ where $\K$ is a field of characteristic zero. In Section \ref{section-rksec},  we study such algebras with arbitrary codimension in terms of their Jordan types. We present an approach to determine the Jordan types of Artinian    Gorenstein algebras using Macaulay duality. We assign a natural invariant to an Artinian    Gorenstein algebra $A$ providing the ranks of multiplication maps by a linear form $\ell$ in different degrees, called \emph{rank matrix}, $M_{\ell,A}$,  Definition \ref{rkmatrix-def}. There is a 1-1 correspondence between rank matrices and so called \emph{Jordan degree types} in Proposition \ref{rkmatrix-1-1-JDT-prop}. We provide necessary conditions for a rank matrix in Lemmas \ref{diffO-seq} and \ref{additiveRank}.
We use this approach in Section  \ref{codim3section} for Artinian Gorenstein algebras in polynomial rings with three variables. We give a complete list of rank matrices that occur for some Artinian Gorenstein algebra $A$ and linear form $\ell$ where $\ell^3=0$ and $\ell^2\neq 0$, see Theorems \ref{3linesHFtheorem-even} and \ref{3linesHFtheorem-odd} for algebras with even and odd socle degrees respectively. As an immediate consequence in Corollary \ref{2linescorollary} we list rank matrices for linear forms where $\ell^2=0$. 
In Theorem \ref{JT-theorem} we prove that the Jordan types of Artinian Gorenstein algebras with codimension three and linear forms $\ell$ where $\ell^4=0$ is uniquely determined by the ranks of at most three multiplication maps, or equivalently, three mixed Hessians.
\section{Preliminaries}
Let $S = \K[x_1,\dots ,x_n]$ be a polynomial  ring equipped with standard grading over a field $\K$ of characteristic zero. Let $A=S/I$ be a graded Artinian  ( its Krull dimension is zero) algebra where $I$ is an homogeneous ideal. 
The \emph{Hilbert function} of a graded Artinian  algebra $A=S/I$ is  a vector of non-negative integers and we denote it by $h_A=(1,h_1,\dots ,h_d)$ where  $h_A(i)=h_i=\dim_{\K}(A_i)$. The integer $d$ is called the \emph{socle degree} of $A$, that is the largest integer $i$ such that $h_A(i)>0$. A  graded  Artinian  algebra $A$ is \emph{  Gorenstein} if $h_d=1$ and its Hilbert function is symmetric, i.e. $h_A(i)=h_A(d-i)$ for $0\leq i\leq d$.\par 
A famous result of F. H. S. Macaulay \cite{Macaulay}  provides a bound on the growth of Hilbert functions of  graded Artinian  algebras. F. H. S. Macaulay characterizes all vectors of non-negative integers that occur as Hilbert functions of standard graded algebras. Such a sequence is called an \emph{O-sequence}. \par 
Let $R=\K[X_1,\dots , X_n ]$ be the Macualay dual ring of $S$. Given a homogeneous ideal $I\subset S$ the \emph{inverse system} of $I$ is defined to be a graded $S$-module $M\subset R$ such that $S$ acts on  $R$ by differentiation. For more details of Macaulay's inverse system see \cite{Geramita} and \cite{IK}.
For graded Artinian    Gorenstein algebras the inverse system is generated by only one form.
\begin{theorem}\cite{MW}\label{dualGen}
 Let $A=S/I$ be a graded Artinian  algebra. Then $A$ is   Gorenstein if and only if there exists a polynomial $F\in R =\K[X_1,\dots ,X_n]$ such that $I=\ann_S(F)$.
 \end{theorem}
From  a  result by F. H. S. Macaulay \cite{F.H.S} it is known that an Artinian standard graded $\mathsf{k}$-algebra $A=S/I$ is Gorenstein if and only if there exists $F\in R_d$, such that $I=\ann_S(F)$.  T. Maeno and J. Watanabe \cite{MW} described higher Hessians of dual generator $F$ and provided a criterion for Artinian Gorenstein algebras having the SLP or WLP.
\begin{definition}\cite[Definition 3.1]{MW}
Let $F$ be a polynomial in $R$ and $A= S/\ann_S(F)$ be its associated Artinian Gorenstein algebra. Let $\mathcal{B}_{j} = \lbrace \alpha^{(j)}_i+\ann_S(F) \rbrace_i$ be a  $\mathsf{k}$-basis of $A_j$. The entries of the  $j$-th Hessian matrix of $F$ with respect to $\mathcal{B}_j$ are given  by
$$
(\Hess^j(F))_{u,v}=(\alpha^{(j)}_u\alpha^{(j)}_v \circ F).
$$
We note that when $j=1$ the form $\Hess^1(F)$ coincides with the usual Hessian. Up to  a non-zero constant multiple  $\det \Hess^j(F)$ is independent of the basis $\mathcal{B}_j$.  By abusing notation we will write   $\mathcal{B}_{j} = \lbrace \alpha^{(j)}_i \rbrace_i$ for a basis of $A_j$.
\end{definition}
R. Gondim and G. Zappal\`a \cite{GZ} introduced a generalization of Hessians which provides the rank of multiplication maps by powers a linear form which are not necessarily symmetric.
\begin{definition}\cite[Definition 2.1]{GZ}
Let $F$ be a polynomial in $R$ and $A= S/\ann_S(F)$ be its associated Gorenstein algebra. Let $\mathcal{B}_{j} = \lbrace \alpha^{(j)}_i\rbrace_i$ and $\mathcal{B}_{k} = \lbrace \beta^{(k)}_i\rbrace_i$ be $\mathsf{k}$-bases of $A_j$ and $A_k$ respectively. The \emph{Hessian matrix of order $(j,k)$} of $F$ with respect to $\mathcal{B}_j$  and $\mathcal{B}_k$ is 
$$
(\Hess^{(j,k)}(F))_{u,v}=(\alpha^{(j)}_u\beta^{(k)}_v \circ F).
$$
When $j=k$, $\Hess^{(j,j)}(F)=\Hess^{j}(F)$. 
\end{definition}
\begin{definition}
 Let $A= S/\ann(F)$ where $F\in R_d$. Pick bases $\B_j = \lbrace \alpha^{(j)}_u\rbrace_u$ and $\B_{d-j} = \lbrace \beta^{(d-j)}_u\rbrace_u$ be $\K$-bases of $A_j$ and $A_{d-j}$ respectively. The \emph{catalecticant matrix of $F$} with respect to $\B_j$ and $\B_{d-j}$ is 
$$
\Cat^j_F=(\alpha^{(j)}_u\beta^{(d-j)}_v F)_{u,v=1}.
$$
\end{definition}
The rank of the $j$-th catalecticant matrix of $F$ is equal to the Hilbert function of $A$ in degree $j$, see \cite[Definition 1.11]{IK}.

We recall the definition of the Jordan degree type for a graded Artinian algebra and linear form.
\begin{definition}{\cite[Definition 2.28]{IMM}}\label{JDT}
Let $A$ be a graded Artinian algebra and $\ell\in A_1$. Suppose that $P_{\ell,A}=(p_1,\dots ,p_t)$ is the Jordan type for $\ell$ and $A$, then there exist elements $z_1, \dots z_t\in A$, which depend on $\ell$, such that $\{\ell^iz_k\mid 1\leq k\leq t, 0\leq i\leq p_k-1\}$ is  a $\mathsf{k}$-basis for $A$. The Jordan blocks of the multiplication map by $ \ell$ is determined by the strings $\mathsf{s}_k=\{z_k, \ell z_k, \dots , \ell^{p_k-1}z_k\}$, and $A$ is the direct sum $A=\langle \mathsf{s}_1\rangle \oplus\dots \oplus  \langle \mathsf{s}_t\rangle$. Denote by $d_k$ the degree of $z_k$. Then the \emph{Jordan degree type},  is defined to be the indexed partition $\mathcal{S}_{\ell,A}=({p_1}_{d_1}, \dots ,{p_t}_{d_t})$.
\end{definition}
\section{Rank matrices for Artinian    Gorenstein algebras of linear forms}\label{section-rksec}
Throughout this section let $S=\K[x_1,\dots, x_n]$ be a polynomial ring with $n\geq 2$ variables equipped with standard grading over a filed  $\K$ of characteristic zero. We let $A=S/\ann(F)$ be a graded Artinian Gorenstein algebra with dual generator $F\in R=\K[X_1,\dots , X_n ]$ that is a homogeneous polynomial of degree $d\geq 2$.
\begin{definition}\label{rkmatrix-def}
Let $A=S/\ann(F)$ be an Artinian    Gorenstein algebra with socle degree $d$. For linear form $\ell\in A$ define the \textit{rank matrix}, $M_{\ell,A}$, of $A$ and $\ell$ to be  the upper triangular square matrix of size  $d+1$ with the  following $i,j$-th entry  
$$(M_{\ell,A})_{i,j} = \rk\left(\times \ell^{j-i} : A_i\longrightarrow A_j\right),$$
for every $i\leq j$. For $i>j$ we set $ (M_{\ell,A})_{i,j}=0$.
\end{definition}

\begin{definition}
Let $A=S/\ann(F)$  be an Artinian    Gorenstein algebra  with socle degree $d$ and linear form $\ell$. For each $0\leq i\leq d$ define the Artinian Gorenstein algebra, $A^{(i)}$, with the dual generator $\ell^i\circ F$
$$
A^{(i)} := S/\ann(\ell^i\circ F).
$$
\end{definition}
\noindent We note that when $i=0$ the algebra $A^{(0)}$ coincides with $A$. \par
\noindent 
\begin{remark}\label{r_ij-remark}
By the definition of higher and mixed Hessians for every $0\leq i<j$ we have that 
\begin{equation}
\rk \Hess^{(i,d-j)}_\ell (F) = (M_{\ell,A})_{i,j}.
\end{equation}
\end{remark} 
For each $0\leq i\leq d$ denote the $i$-the diagonal vector of $M_{\ell,A}$ by $\diag(i,M_{\ell,A})$,
$$\diag(i,M_{\ell,A}):=((M_{\ell,A})_{0,i},(M_{\ell,A})_{1,i+1},\dots ,(M_{\ell,A})_{d-i,d}).$$
We show that for every $0\leq i\leq d$ the vector $\diag(i,M_{\ell,A})$ is the Hilbert function of some Artinian    Gorenstein algebra.
We denote the Macaulay inverse system module of $A=S/\ann(F)$ by $\langle F\rangle$.
\begin{proposition}\label{diagprop}
Let $A=S/\ann(F)$ be an Artinian    Gorenstein algebra with socle degree $d\geq 2$ and $\ell$ be a linear form. Then
$$
\diag(i,M_{\ell,A}) = h_{A^{(i)}},
$$
for every $0\leq i\leq d$.
\end{proposition}
\begin{proof}
By the definition of rank matrix $M_{\ell,A}$ we have that the entries on the $i$-th diagonal of $M_{\ell,A}$ are exactly the ranks of multiplication map by $\ell^i$ on $A$ in various degrees.
Using Macaulay duality for every $0\leq j\leq \lfloor\frac{d-i}{2}\rfloor$ we get the following 
\begin{align*}
rk\left(\times \ell^{i} : A_j\longrightarrow A_{i+j}\right)&=rk\left(\circ \ell^{i} : \langle F\rangle_{i+j}\longrightarrow \langle F\rangle_{j}\right)\\&=\dim_{\K} \langle \ell^i\circ F \rangle_j\\
&=\dim_{\K}(S/\ann(\ell^i\circ F))_{j}.
\end{align*}%\footnote{check it}
Note that the socle degree of $A^{(i)}$ is equal to $d-i$. The proof is complete since $h_{A^{(i)}}$ is symmetric about $\lfloor\frac{d-i}{2}\rfloor$.
\end{proof}
\begin{example}\label{firstEx}
Let $A=\K[x_1,x_2,x_3]/\ann(F)$ be Artinian Gorenstein algebra where $F=X_1^2X_2^2X_3^2$. We have that 
$h_A=\left(1,3,6,7,6,3,1\right)$. Consider $\ell=x_1$, then
$$h_{A^{(1)}}=h_{S/\ann(x_1\circ F)}=\left(1,3,5,5,3,1\right),\quad 
h_{A^{(2)}}=h_{S/\ann(x_1^2\circ F)}=\left(1,2,3,2,1\right),
$$
and $x_1^i\circ F=0$ for  $i\geq 3$.  Then the rank matrix is as follows 
$$M_{x_1,A}= \begin{pmatrix}
1&1&1&0&0&0&0\\
0&3&3&2&0&0&0\\
0&0&6&5&3&0&0\\
0&0&0&7&5&2&0\\
0&0&0&0&6&3&1\\
0&0&0&0&0&3&1\\
0&0&0&0&0&0&1\\
\end{pmatrix}.
$$
By Remark \ref{r_ij-remark} we have that
\begin{align*}
&\rk \Hess_{x_1}^{(0,5)}=\rk \Hess_{x_1}^{(0,4)}=1, \rk \Hess_{x_1}^{(1,4)}=3,\rk \Hess_{x_1}^{(1,3)}=2,\\
 & \rk \Hess_{x_1}^{(2,3)}=5,\hspace*{2mm}\text{and}\hspace*{2mm} \rk \Hess_{x_1}^{(2,2)}=3.
\end{align*}
\end{example}

Te following two lemmas provide conditions on every rank matrix $M_{\ell,A}$.
First  we set a notation. For a vector $\mathbf{v}$ of positive integers of length $l$ denote by $\mathbf{v}_+$ the vector of length $l+1$ obtained by adding zero to vector $\mathbf{v}$, that is $\mathbf{v}_+ = (0,\mathbf{v})$.

\begin{lemma}\label{diffO-seq}
For every $0\leq i\leq d-1$, the difference vector $ h_{A^{(i)}}-(h_{A^{(i+1)}})_+$ is an O-sequence.
\end{lemma}
\begin{proof}% \footnote{more detailed proof}
Using Macaulay duality, for every $j\geq 1$ we have 
\begin{align*}
& h_{A^{(i)}}(j)-h_{A^{(i+1)}}(j-1) = \dim_{\K}\langle \ell^i\circ F\rangle_j -\dim_{\K}\langle \ell^{i+1}\circ F\rangle_{j-1}=\dim_{\K}\left(\langle \ell^i\circ F\rangle/\langle \ell^{i+1}\circ F\rangle\right)_{j}.
\end{align*}
For $j=0$ we have that $\dim_{\K}\left(\langle \ell^i\circ F\rangle/\langle \ell^{i+1}\circ F\rangle\right)_{0}=1$ if ${A^{(i)}}\neq 0$. If ${A^{(i)}}=0$ then clearly ${A^{(i+1)}}= 0$ and so  $ h_{A^{(i)}}-(h_{A^{(i+1)}})_+$ is the zero vector.\par 
\noindent  We conclude that $h_{A^{(i)}}-(h_{A^{(i+1)}})_+$ is the Hilbert function of $\left(\langle \ell^i\circ F\rangle/\langle \ell^{i+1}\circ F\rangle\right)$, and hence it is an O-sequence.
\end{proof}
\begin{lemma}\label{additiveRank}
For every $i,j\geq 1$, the following inequality holds 
$$
h_{A^{(i-1)}}(j)+h_{A^{(i+1)}}(j-1)\geq h_{A^{(i)}}(j)+h_{A^{(i)}}(j-1).
$$
\end{lemma}
\begin{proof}
The inclusion map $\langle \ell^{i+1}\circ F\rangle\hookrightarrow \langle \ell^{i}\circ F\rangle$ for every $i\geq 0$ induces the following commutative diagram
$$
\xymatrix{
0\ar[r]& \langle \ell^{i+1}\circ F\rangle \ar[d]\ar[r]&\langle \ell^{i}\circ F\rangle  \ar[d]\ar[r]& \langle \ell^{i}\circ F\rangle/ \langle \ell^{i+1}\circ F\rangle \ar[d]^{\varphi}\ar[r]&0\\
0\ar[r]& \langle \ell^i\circ F\rangle\ar[r]&\langle \ell^{i-1}\circ F\rangle\ar[r]& \langle \ell^{i-1}\circ F\rangle/ \langle \ell^i\circ F\rangle\ar[r]&0\\
}
$$
which shows that $\varphi$ is also injective.
Using Lemma \ref{diffO-seq} we get that $h_{A^{(i)}}(j)-h_{A^{(i+1)}}(j-1) = \dim_{\K}\left(\langle \ell^i\circ F\rangle/\langle \ell^{i+1}\circ F\rangle\right)_{j},$ for every $i,j\geq 1$ that implies the desired inequality.
\end{proof}
\begin{remark}
The above lemma shows that for every $i,j\geq 1$ the following inequality holds
$$
\rk \Hess_\ell^{(j,d-i-j+1)}+\rk \Hess_\ell^{(j-1,d-i-j)}\geq \rk \Hess_\ell^{(j,d-i-j)}+\rk \Hess_\ell^{(j-1,d-i-j+1)}.
$$
\end{remark}
As a consequence of the above lemmas, we provide necessary conditions for an upper triangular square matrix of size $d+1$ with non-negative integers to occur for an Artinian Gorenstein algebra $A$ and linear form $\ell\in A_1$. 
\begin{corollary}\label{cor-rkmatrix}
Let $M$ be an upper triangular matrix of size $d+1$ with non-negative entries. Then $M$ is the rank matrix of some Artinian Gorenstein algebra $A$ and linear form $\ell\in A_1$, only if the following conditions are satisfied.
\begin{itemize}
\item[$(i)$] For every $0\leq i\leq d$, $\diag(i,M)$ is an O-sequences, and $h_A=\diag(0,M)$;
\item[$(ii)$] for every $0\leq i\leq d-1$, the difference vector $\diag(i,M)-\left(\diag(i+1,M)\right)_+$ is an O-sequences;
\item[$(iii)$] for any $2\times 2$ square submatrix of successive entries on and above the diagonal of $M$ of the form $\begin{pmatrix}
u&v\\
w&z\\
\end{pmatrix}$ we have that $w+v\geq u+z$.
\end{itemize}
\end{corollary} 
\begin{proof}
It is an immediate consequence of Proposition \ref{diagprop} and Lemmas \ref{diffO-seq} and \ref{additiveRank}.
\end{proof}
\begin{example}
The following matrix does not occur as the rank matrix of some Artinian Gorenstein algebra and linear form $\ell$.
$$ M=\begin{pmatrix}
1&1&1&0&0&0\\
0&3&2&2&0&0\\
0&0&3&3&2&0\\
0&0&0&3&2&1\\
0&0&0&0&3&1\\
0&0&0&0&0&1\\
\end{pmatrix}.
$$
Since condition $(ii)$ in Corollary \ref{cor-rkmatrix} is not satisfied; in fact $\diag(0,M)-(\diag(1,M))_+=(1,3,3,3,3,1)-(0,1,2,3,2,1)=(1,2,1,0,1,0)$ is not an O-sequence.\par 
Corollary \ref{cor-rkmatrix} also implies that the following matrix is not a possible rank matrix for some $A$ and $\ell$.
$$ N=\begin{pmatrix}
1&1&1&0&0&0\\
0&3&3&1&0&0\\
0&0&5&4&1&0\\
0&0&0&5&3&1\\
0&0&0&0&3&1\\
0&0&0&0&0&1\\
\end{pmatrix}.
$$
In fact, for submatrix $\begin{pmatrix} 3&1\\5&4\end{pmatrix}$ the condition $(iii)$ is not satisfied.
%we have 
%\begin{align*}
%\diag(0,N)(1,3,5,5,3,1)-(0,1,3,4,3,1)=(1,2,2,1)\ngeqq (1,2,3,2)= (1,3,4,3,1)-(0,1,1,1,1).
%\end{align*}
\end{example}
\begin{definition}[Jordan degree type matrix] Let $A=S/\ann(F)$ be an Artinian    Gorenstein algebra and $\ell\in A$ a linear form. Assume that $M_{\ell,A}$ is the rank matrix of $A$ and $\ell$. We define
the \emph{Jordan degree type matrix}, $J_{\ell,A}$, of $A$ and $\ell$ to be the upper triangular matrix with the following non-negative entries 
\begin{align}\label{J(A,l)definition}
(J_{\ell,A})_{i,j} :
=& (M_{\ell,A})_{i,j}+(M_{\ell,A})_{i-1,j+1}-(M_{\ell,A})_{i-1,j}-(M_{\ell,A})_{i,j+1},
\end{align}
where we set $(M_{\ell,A})_{i,j}$ if either  $i< 0$ or $j< 0$.
%Since $M_{\ell,A}$ is upper triangular, $J_{\ell,A}$ is also upper triangular. For each $0\leq i\leq j$ we set $k=j-i$, then 
\begin{equation}\label{JDT_ij}
(J_{\ell,A})_{i,j} = h_{A^{(k)}}(i)+h_{A^{(k+2)}}(i-1)-h_{A^{(k+1)}}(i-1)-h_{A^{(k+1)}}(i),
\end{equation}
such that for every $k$, $h_{A^{(k)}}(-1):=0$.
\end{definition}
Recall from Lemma \ref{additiveRank} that for each $0\leq i\leq j$, $(J_{\ell,A})_{ij}$ is non-negative.
\begin{proposition}\label{rkmatrix-1-1-JDT-prop}
There is a 1-1 correspondence between the two matrices $M_{\ell,A}$ and $J_{\ell,A}$ associated to a pair $(A,\ell)$.
\end{proposition}
\begin{proof}

We use Equation (\ref{J(A,l)definition}) to provide an algorithm to obtain $J_{\ell,A}$ from $M_{\ell,A}$. For each $0\leq i\leq j$ define matrix $J^\prime_{\ell,A}$ as the following 
\begin{equation}\label{J'matrixdef}
(J^\prime_{\ell,A})_{i,j} := (M_{\ell,A})_{i,j}-(M_{\ell,A})_{i,j+1},
\end{equation}
where we set $(M_{\ell,A})_{i,j}=0$ if either $i< 0$ or $j< 0$.
Then define the upper triangular matrix $J_{\ell,A}$ where its entry $i,j$ for every $0\leq i\leq j$ is equal to 
\begin{equation}\label{JfromJ'def}
(J_{\ell,A})_{i,j} = (J^\prime_{\ell,A})_{i,j}-(J^\prime_{\ell,A})_{i-1,j},
\end{equation}
where we set $(J^\prime_{\ell,A})_{i,j}=0$  if either $i< 0$ or $j< 0$.

We obtain  $M_{\ell,A}$ from $J^\prime_{\ell,A}$ in two steps.
First we get the matrix $J^\prime_{\ell,A}$ from $J_{\ell,A}$. For each $0\leq i\leq j$,  we have the following
\begin{equation}
(J^\prime_{\ell,A})_{i,j} = (J_{\ell,A})_{i,j}+(J_{\ell,A})_{i-1,j},
\end{equation}
where we set $(J_{\ell,A})_{i,j}=0$ if either $i< 0$ or $j< 0$.
Then for each $0\leq i\leq j$, 
\begin{equation}
(M_{\ell,A})_{i,j}=(J^\prime_{\ell,A})_{i,j}+(J^\prime_{\ell,A})_{i,j+1},
\end{equation}
where we set $(J^\prime_{\ell,A})_{i,j}=0$ if either $i< 0$ or $j< 0$.
\end{proof}
\begin{example}
We illustrate the procedure provided in Proposition \ref{rkmatrix-1-1-JDT-prop} for the Artinian  Gorenstein algebra given in Example \ref{firstEx} with the rank matrix $M_{\ell,A}$. Using Equations (\ref{J'matrixdef}) and  (\ref{JfromJ'def}) we get the following matrices.
$$M_{\ell,A}= \begin{pmatrix}
1&1&1&0&0&0&0\\
0&3&3&2&0&0&0\\
0&0&6&5&3&0&0\\
0&0&0&7&5&2&0\\
0&0&0&0&6&3&1\\
0&0&0&0&0&3&1\\
0&0&0&0&0&0&1\\
\end{pmatrix},
\hspace*{0mm} J^\prime_{\ell,A} = \begin{pmatrix}
0&0&1&0&0&0&0\\
0&0&1&2&0&0&0\\
0&0&1&2&3&0&0\\
0&0&0&2&3&2&0\\
0&0&0&0&3&2&1\\
0&0&0&0&0&2&1\\
0&0&0&0&0&0&1\\
\end{pmatrix}, \hspace*{0mm} J_{\ell,A} = \begin{pmatrix}
0&0&1&0&0&0&0\\
0&0&0&2&0&0&0\\
0&0&0&0&3&0&0\\
0&0&0&0&0&2&0\\
0&0&0&0&0&0&1\\
0&0&0&0&0&0&0\\
0&0&0&0&0&0&0\\
\end{pmatrix}. $$
\end{example}
Define the decreasing sequence  $\mathbf{d}:=(\dim_{\mathsf{k}}A^{(0)}, \dim_{\mathsf{k}}A^{(1)}, \dots , \dim_{\mathsf{k}}A^{(d)})$, and recall that the second difference sequence of $\mathbf{d}$ is denoted by $\Delta^2 \mathbf{d}$ and its  $i$-th entry is given by $$\Delta^2 \mathbf{d}(i)=\dim_{\mathsf{k}}A^{(i)}+\dim_{\mathsf{k}}A^{(i+2)}-2\dim_{\mathsf{k}}A^{(i+1)},$$
where we set $\dim_{\mathsf{k}}A^{(i)}=0$ for $i>d$.
\begin{proposition}\label{JT}
Let $A=S/\ann(F)$ be an Artinian  Gorenstein algebra with socle degree $d\geq 2$ and let $\ell\in A$ be a linear form. Then the Jordan type partition of $\ell$ for $A$ is given by
$$
P_{\ell,A} = \big(\underbrace{d+1,\dots ,d+1}_{n_{d}} ,\underbrace{d,\dots ,d}_{n_{d-1}} ,\dots ,\underbrace{2,\dots ,2}_{n_1},\underbrace{1,\dots ,1}_{n_0}\big),
$$
such that $\mathbf{n}=(n_0,n_{1}, \dots , n_d)=\Delta^2 \mathbf{d}.$
\end{proposition}
\begin{proof}
The Jordan type partition of $\ell$ for $A$ is equal to the dual partition of the following partition 
\begin{equation}\label{jordantype}
\Big(\rk (\times \ell^0)-\rk (\times \ell^1),\rk (\times \ell^1)-\rk (\times \ell^2),\dots , \rk (\times \ell^{d-1})-\rk (\times \ell^d),\rk (\times \ell^d)\Big).
\end{equation}
Since for each $0\leq i\leq d$ the rank of the multiplication map $ \times \ell^{i}:A_j\longrightarrow A_{j+i}$ is equal to the rank of differentiation map  $\circ \ell^i :\langle F \rangle_{i+j}\longrightarrow\langle F\rangle_{j}$, where $\langle F\rangle$  is the dual algebra to $A$. Thus the rank of $ \times \ell^i:A_j\longrightarrow A_{j+i}$ is equal to $\dim_{\K}\left( S/\ann(\ell^i\circ F)\right)_j$ and therefore we have 
$$\rk \left( \times \ell^i: A\longrightarrow A\right) = \sum^{d-i}_{j=0} \dim_{\K}\left( S/\ann(\ell^i\circ F)\right)_j=\dim_{\mathrm{k}}A^{(i)}.$$
So (\ref{jordantype}) is equal to the following partition
$$
\big(\dim_{\mathrm{k}}A^{(0)}-\dim_{\mathrm{k}}A^{(1)},\dim_{\mathrm{k}}A^{(1)}-\dim_{\mathrm{k}}A^{(2)},\dots ,\dim_{\mathrm{k}}A^{({d-1})}-\dim_{\mathrm{k}}A^{(d)},\dim_{\mathrm{k}}A^{(d)}\big).
$$
The dual partition to the above partition is the Jordan type partition of $A$ and $\ell$ as we claimed.
\end{proof}

\begin{example}
Consider the Artinian    Gorenstein algebra given in Example \ref{firstEx} and linear form $\ell=x_1$. The Jordan degree type matrix of $A$ and $\ell$ is equal to the following matrix 
$$J_{\ell,A} = \begin{pmatrix}
0&0&1&0&0&0&0\\
0&0&0&2&0&0&0\\
0&0&0&0&3&0&0\\
0&0&0&0&0&2&0\\
0&0&0&0&0&0&1\\
0&0&0&0&0&0&0\\
0&0&0&0&0&0&0\\
\end{pmatrix}.
$$
We have that  $P_{\ell,A} = (\underbrace{3, \dots ,3}_9)$. In order to obtain the Jordan degree type $\mathcal{S}_{\ell,A}$ we recall the Definition \ref{JDT} and note that the degree of each part in $P_{\ell,A}$ is equal to the row index of the corresponding entry in $J_{\ell,A}$, so  $\mathcal{S}_{\ell,A}=(3_0,3_1,3_1,3_2,3_2,3_2,3_3,3_3,3_4)$.
\end{example}

\begin{remark}
Equation (\ref{JDT_ij}) may be expressed in terms of the mixed Hessians.
\begin{small}
\begin{equation}
(J_{\ell,A})_{i,j} = \rk\Hess_\ell^{(i,d-i-k)}(F)+\rk\Hess_\ell^{(i-1,d-i-k-1)}(F)-\rk\Hess_\ell^{(i-1,d-i-k)}(F)-\rk\Hess_\ell^{(i,d-i-k-1)}(F).
\end{equation}
\end{small}
This recovers a result by R. Gondim and B. Costa \cite[Theorem 4.7]{CG} determining Jordan types of  Artinian Gorenstein algebras and linear forms using the ranks of mixed Hessians. 
\end{remark}

\section{Jordan types of Artinian  Gorenstein algebras of codimension three}\label{codim3section}
In this section we consider graded Artinian Gorenstein quotients of $S=\K[x,y,z]$ where $\mathrm{char}(\K)=0$. For an  Artinian  Gorenstein algebra  $A=S/\ann(F)$ with dual generator $F\in R=\mathsf{k}[X,Y,Z]$ of degree $d\geq 2$ and a linear form $\ell$ we explain how we find the rank matrix $M_{\ell,A}$, and as a consequence the Jordan type $P_{\ell,A}$.  

Let $L_1, L_2,L_3$ be linear forms in the dual ring $R=\mathsf{k}[X,Y,Z]$  such that $\ell\circ L_1\neq 0$ and $\ell\circ L_2=\ell\circ L_3 = 0$.
By linear change of coordinates we may assume that $L_1=X$, $L_2=Y$ and $L_3=Z$. Then $F$ can be written in the following form 
$$
F = \sum_{i=0}^dX^iG_{d-i},
$$
where for each $0\leq i\leq d$, $G_{d-i}$ is a homogeneous polynomial of degree $d-i$ in the variables $Y$ and $Z$. In general $G_{d-i}$ could be a zero polynomial for some $i$.

\subsection{Jordan types with parts of length at most four}\label{length3}

We will provide the list of all possible rank matrices $M_{\ell,A}$ such that $A$ is an Artinian Gorenstein algebra and $\ell$ is a linear form in $A$ where $\ell^3=0$. 
Assuming  $\ell^3=0$ implies that $M_{\ell,A}$ has at most three  non-zero diagonals. Consequently, we provide a formula to compute the  Jordan type partitions for Artinian Gorenstein algebras and linear forms $\ell$ such that $\ell^4=0$, which are Jordan types with parts of length at most four.

Consider Artinian  Gorenstein algebra $A=S/\ann(F)$ with socle degree $d\geq 2$ and linear form $\ell$ such that $\ell^3=0$. Without loss of generality we assume that $\ell=x$ and that $F$ is in the following form
\begin{equation}\label{F}
F = X^2G_{d-2}+XG_{d-1}+G_d,
\end{equation}
where
$$G_{d}=\sum_{j=0}^{d}\frac{a_{j}}{j!(d-j)!}{Y^{d-j}Z^{j}}, \quad  G_{d-1}=\sum_{j=0}^{d-1}\frac{b_{j}}{j!(d-j-1)!}{Y^{d-j-1}Z^{j}},$$
and 
$$ G_{d-2}=\frac{1}{2}\sum_{j=0}^{d-2}\frac{c_{j}}{j!(d-j-2)!}{Y^{d-j-2}Z^{j}}.$$
In order to make the computations simpler, we choose the coefficients of the terms in $F$ in a way that the entries of the catalecticant matrices of $F$ are either zero or one.

We first consider the case when $\ell^3=0$ but $\ell^2\neq 0$. Therefore, we assume that  $G_{d-2}\neq 0$ since otherwise we get $\ell^2=0$. Recall that $A^{(0)}=A$, $A^{(1)}=S/\ann(\ell\circ F)$, $A^{(2)}=S/\ann(\ell^2\circ F)$ and $A^{(i)}=S/\ann(\ell^i\circ F)=0$, for every $i\geq 3$.

We determine all rank matrices that occur for such algebras and linear forms $\ell$ where $\ell^3=0$. Equivalently, we determine all possible Hilbert functions for $A, A^{(1)}$ and $A^{(2)}$. The rank matrices are slightly different for even  and odd socle degrees, as excepted, thus we treat these cases separately. We first prove our result for Artinian Gorenstein algebras with even socle degree $d\geq 2$. Later, in  similar cases for odd socle degrees we refer to the relevant proof given for even socle degrees. 

We will show in the theorems bellow that the rank matrix, $M_{\ell,A}$, for $\ell^3=0$ and $A$ is determined by three of its entries. These entries are exactly  the maximum values in non-zero diagonals of $M_{\ell,A}$, that are maximum values of $h_A$, $h_{A^{(1)}}$ and $h_{A^{(2)}}$. The maximum value of the Hilbert function of an Artinian Gorenstein algebra is obtained in the middle degree. We denote by $r,s$ and $t$ the maximum value for the Hilbert function of $h_{A^{(2)}}, h_{A^{(1)}}$ and $h_A$ respectively. We first provide all possible triples $(r,s,t)$.

\begin{lemma}[Even socle degree]\label{maxvaluesevenLemma}
There exists an Artinian Gorenstein algebra $A$ with even socle degree $d\geq 2$ and linear form $\ell\in A_1$ where $\ell^2\neq 0$ but $\ell^3=0$, such that $$(r,s,t)=(h_{A^{(2)}}(\frac{d}{2}-1) ,h_{A^{(1)}}(\frac{d}{2}-1),h_A(\frac{d}{2}))$$ if and only if 
\begin{itemize}
\item [$(1)$]  $r\in [1,\frac{d}{2}-1]$, $s\in[2r,\frac{d}{2}+r]$ and $t\in[2s-r,\frac{d}{2}+s+1]$, for $d\geq 4$;  or
\item [$(2)$] $r=\frac{d}{2}$, $s=d-1$ and $t\in[\frac{3d}{2}-2,\frac{3d}{2}]$,  for $d\geq 2$.  
\end{itemize}
\end{lemma}
\begin{proof}
We prove the statement by analysing the catalecticant matrices in the desired degrees. In each case we first determine all possible ranks for each catalecticant matrix and then for each possible value we provide polynomials $G_{d-2}, G_{d-1}$ and $G_d$ as in (\ref{F}) giving the certain ranks. \par 
The maximum value of the  Hilbert function of $A$ occurs in degree $\frac{d}{2}$ and it is equal to $\rk\Cat_F(\frac{d}{2})$. Pick the following monomial basis for $A_{\frac{d}{2}}$
$$
\B_{\frac{d}{2}} = \{x^{\frac{d}{2}},x^{\frac{d}{2}-1}y,x^{\frac{d}{2}-1}z,x^{\frac{d}{2}-2}y^2,x^{\frac{d}{2}-2}yz,x^{\frac{d}{2}-2}z^2,\dots , y^{\frac{d}{2}},y^{\frac{d}{2}-1}z,\dots ,z^{\frac{d}{2}}\}.
$$
Then the catalecticant matrix of $F$ with respect to $\B_{\frac{d}{2}}$ is equal to 
\begin{equation}
\Cat_F(\frac{d}{2})=\left[\begin{array}{@{}c|c|c@{}}
\mathbf{0}& \mathbf{0}&{\Cat_{G_{d-2}}{(\frac{d}{2}-2)}}
\\\hline
\mathbf{0}&\Cat_{G_{d-2}}{(\frac{d}{2}-1)}&\Cat_{G_{d-1}}{(\frac{d}{2}-1)}\\\hline
\Cat_{G_{d-2}}{(\frac{d}{2})}&\Cat_{G_{d-1}}{(\frac{d}{2})}&\Cat_{G_{d}}{(\frac{d}{2})}
\end{array}
\right].\\
\end{equation}
Which is equal to 
\begin{equation}\label{catmatrix3lines}
 \Cat_F(\frac{d}{2}) = \left[\begin{array}{@{}cccc|cccc|cccc@{}}
    0 & 0 & \cdots & 0 &  0 & 0 & \cdots & 0 & c_{0} & c_{1}&\cdots & c_{\frac{d}{2}} \\
     0 & 0 & \cdots & 0 & 0 & 0 &\cdots & 0 & c_{1} & c_{2} &\cdots & c_{\frac{d}{2}+1} \\
    \vdots & \vdots & \reflectbox{$\ddots$}  &\vdots & \vdots & \vdots & \reflectbox{$\ddots$}  &\vdots  & \vdots  &  \vdots & \reflectbox{$\ddots$} &\vdots\\
 0 & 0 & \cdots & 0 &  0 & 0 &\cdots & 0 &  c_{\frac{d}{2}-2} & c_{\frac{d}{2}-1} &\cdots & c_{d-2} \\\hline
  0 & 0 & \cdots & 0 &  c_{0} & c_{1}&\cdots & c_{\frac{d}{2}-1} & b_{0} & b_{1}&\cdots & b_{\frac{d}{2}} \\
 0 & 0 & \cdots & 0 &   c_{1} & c_{2} &\cdots & c_{\frac{d}{2}}  & b_{1} & b_{2} &\cdots & b_{\frac{d}{2}+1} \\
    \vdots  &\vdots & \reflectbox{$\ddots$}  &  \vdots&   \vdots  &\vdots & \reflectbox{$\ddots$}  &  \vdots& \vdots  &  \vdots & \reflectbox{$\ddots$} &\vdots\\
  0 & 0 & \cdots & 0 &  c_{\frac{d}{2}-1} & c_{\frac{d}{2}} &\cdots & c_{d-2} & b_{\frac{d}{2}-1} & b_{\frac{d}{2}} &\cdots & b_{d-1} \\\hline
     c_{0} & c_{1}&\cdots & c_{\frac{d}{2}-2} & b_{0} & b_{1}&\cdots & b_{\frac{d}{2}-1} &a_{0} & a_{1}&\cdots & a_{\frac{d}{2}} \\
 c_{1} & c_{2} &\cdots & c_{\frac{d}{2}-1}  & b_{1} & b_{2} &\cdots & b_{\frac{d}{2}}& a_{1} & a_{2} &\cdots & a_{\frac{d}{2}+1} \\
     \vdots  &\vdots & \reflectbox{$\ddots$}  &  \vdots& \vdots  &  \vdots & \reflectbox{$\ddots$} &\vdots& \vdots  &  \vdots & \reflectbox{$\ddots$} &\vdots\\
    c_{\frac{d}{2}} & c_{\frac{d}{2}+1} &\cdots & c_{d-2} & b_{\frac{d}{2}} & b_{\frac{d}{2}+1} &\cdots & b_{d-1}& a_{\frac{d}{2}} & a_{\frac{d}{2}+1} &\cdots & a_{d} \\
     \end{array}\right].
\end{equation}
Since any Artinian  algebra of codimension two has the SLP  the rank of the $j$-th Hessian matrices of  polynomials $G_{d-2},G_{d-1}$ and $G_d$ are equal to the ranks of their $j$-th catalecticant matrices. By linear change of coordinates, we may assume that $z$ is the strong Lefschetz element for Artinian  Gorenstein algebra $\mathsf{k}[y,z]/\ann(G_{d-2})$. This implies that the lower right square submatrices of the catalecticant matrices of $G_{d-2}$ in all degrees have maximal rank. Likewise, we may assume that $y$ is the strong Lefschetz element for the Artinian Gorenstein algebra $\mathsf{k}[y,z]/\ann(G_{d-1})$ which means that the upper left square submatrices of the catalecticant matrices of $G_{d-1}$ in different degrees are all full rank. \par 
Observe that $r=h_{A^{(2)}}(\frac{d}{2}-1)\in [1,\frac{d}{2}]$. To show $(1)$ we assume $r=h_{A^{(2)}}(\frac{d}{2}-1)\in [1,\frac{d}{2}-1]$ which implies that $
 h_{A^{(2)}}(\frac{d}{2}-2)= h_{A^{(2)}}(\frac{d}{2}-1)= h_{A^{(2)}}(\frac{d}{2})=r.
$ We assume that the ranks of the lower right submatrices of $\Cat_{G_{d-2}}(\frac{d}{2}-2), \Cat_{G_{d-2}}(\frac{d}{2}-1)$ and $\Cat_{G_{d-2}}(\frac{d}{2})$ are equal to $r$ and setting $c_{d-r-1}=1$ and $c_{i}=0$ for every $i\neq d-r-1 $ provides the desired property. So
\begin{equation}\label{G_(d-2)even(1)}
G_{d-2}= \frac{Y^{r-1}Z^{d-r-1}}{(r-1)!(d-r-1)!}, \quad \text{for all}\quad r\in[1,\frac{d}{2}-1].
\end{equation}
Now in order to obtain possible values for $s=h_{A^{(1)}}(\frac{d}{2}-1)$, we notice that $s\in [2r, 2r+\rk \mathbf{B}]$ where $\mathbf{B}$ is the following matrix 
$$
\mathbf{B}=\left(\begin{array}{@{}ccccccc@{}}
 b_{0} &\cdots & b_{\frac{d}{2}-r} \\
    \vdots  & \reflectbox{$\ddots$} &\vdots\\
b_{\frac{d}{2}-1-r}&\cdots & b_{d-1-2r}      \end{array}\right).
$$
Since the socle degree of $A^{(1)}$  is equal to $d-1$ that is an odd integer, we get that $ h_{A^{(1)}}(\frac{d}{2}-1)= h_{A^{(1)}}(\frac{d}{2}) = s$. For every $s\in [2r, 2r+\rk\mathbf{B}]$, we have $\rk\mathbf{B}=s-2r$. We may assume that the upper left submatrix of $\mathbf{B}$ has rank $s-2r$. Setting  $G_{d-1}=0$ provides that $\rk\mathbf{B}=s-2r=0$. And setting $b_{s-2r-1}=1$ and $b_i=0$ for every $i\neq s-2r-1$ implies that $\rk\mathbf{B}=s-2r\neq 0$.  Equivalently, we set 
\begin{equation}\label{G_(d-1)even(1)}
G_{d-1}=\left\{
                \begin{array}{ll}
                  0 & \text{if $s-2r=0$},\\
                  \frac{Y^{d-s+2r}Z^{s-2r-1}}{(d-s+2r)!(s-2r-1)!} & \text{if $1\leq s-2r\leq \frac{d}{2}-r.$}\\
                \end{array}
              \right.
\end{equation}
This implies that, there exists $A$ such that $h_{A^{(1)}}(\frac{d}{2}-1)=s $ if and only if  $s\in [2r,\frac{d}{2}+r].$\par 
To obtain possible values for $t=h_A(\frac{d}{2})$, first notice that $t\in [2s-r, 2s-r+\rk\mathbf{A}]$, for
$$
\mathbf{A}=\left(\begin{array}{@{}ccccccc@{}}
 a_{2s-4r} &\cdots & a_{\frac{d}{2}-3r+s} \\
    \vdots  & \reflectbox{$\ddots$} &\vdots\\
a_{\frac{d}{2}-3r+s}&\cdots & a_{d-2r}      \end{array}\right).
$$
For every $t\in [2s-r, 2s-r+\rk\mathbf{A}]$, we have that $\rk\mathbf{A}=t-2s+r$. We may assume that the rank of the upper left submatrix of $\mathbf{A}$ is equal to $t-2s+r$. 
%Now consider the following submatrix of the matrix in (\ref{catmatrix3lines}), for every $j\in [0,\frac{d}{2}-r]$
%$$
%\mathbf{A}=\left(\begin{array}{@{}ccccccc@{}}
 %a_{2j} &\cdots & a_{\frac{d}{2}-r+j} \\
  %  \vdots  & \reflectbox{$\ddots$} &\vdots\\
%a_{\frac{d}{2}-r+j}&\cdots & a_{d-2r}      \end{array}\right).
%$$
For $G_d=0$ we get $\rk\mathbf{A}=t-2s+r=0$.  Setting $a_{t-3r-1}=1$ and $a_i=0$ for every $i\neq t-3r-1$ provides that  $\rk\mathbf{A}=t-2s+r\neq 0$. In other words, we choose $G_d$ as the following
\begin{equation}\label{G_(d)even(1)}
G_{d}=\left\{
                \begin{array}{ll}
                  0 & \text{if $t-2s+r=0$},\\
                  \frac{Y^{d-t+3r+1}Z^{t-3r-1}}{(d-t+3r+1)!(t-3r-1)!} & \text{if $1\leq t-2s+r\leq \frac{d}{2}+r-s+1.$}\\
                \end{array}
              \right.
\end{equation}
So there exists $A$ such that $t=h_A(\frac{d}{2})$ if and only if 
$t\in [2s-r,\frac{d}{2}+s+1]$.
%$$\quad h_{A}(\frac{d}{2}-1) = 3r+2j+k,\quad \text{for all}\hspace*{2mm} j\in [0,\frac{d}{2}-r],\hspace*{2mm}\text{and}\hspace*{2mm}k\in [0,\frac{d}{2}-r-j+1],$$
%or equivalently,
%$$\quad h_{A}(\frac{d}{2}-1) = t,\quad \text{for all}\hspace*{2mm} t\in [2s-r,\frac{d}{2}+s+1].$$

To prove $(2)$ assume $h_{A^{(2)}}(\frac{d}{2}-1)=\frac{d}{2}$. This implies that the Hilbert function of $h_{A^{(2)}}$ has the maximum possible value up to degree $\frac{d}{2}-1$ and since the socle degree of $A^{(2)}$ is even and is equal to $d-2$ we have  $$h_{A^{(2)}}(\frac{d}{2}-2)=h_{A^{(2)}}(\frac{d}{2})=\frac{d}{2}-1.$$ So setting $c_{\frac{d}{2}-1}=1$ and $c_i=0$ for every $i\neq \frac{d}{2}-1$, or equivalently,
\begin{equation}\label{Gd-2maxeven}
G_{d-2} = \frac{Y^{\frac{d}{2}-1}Z^{\frac{d}{2}-1}}{(\frac{d}{2}-1)!(\frac{d}{2}-1)!}
\end{equation}
provides  the desired ranks for the catalecticant matrices $\Cat_{G_{d-2}}(\frac{d}{2}-2)$, $\Cat_{G_{d-2}}(\frac{d}{2}-1)$ and $\Cat_{G_{d-2}}(\frac{d}{2})$. \\ 
We have that $h_{{A^{(1)}}}(\frac{d}{2}-1)=\rk\Cat_{x\circ F}(\frac{d}{2}-1)$, and 
\begin{equation}
\Cat_{x\circ F}(\frac{d}{2}-1) = \left[\begin{array}{@{}c|c@{}}
 \mathbf{0}&{\Cat_{G_{d-2}}{(\frac{d}{2}-2)}}
\\\hline
\Cat_{G_{d-2}}{(\frac{d}{2}-1)}&\Cat_{G_{d-1}}{(\frac{d}{2}-1)}\\
\end{array}
\right].\\
\end{equation}
Since $\rk {\Cat_{G_{d-2}}{(\frac{d}{2}-2)}}=\frac{d}{2}-1$ and $\rk\Cat_{G_{d-2}}{(\frac{d}{2}-1)}=\frac{d}{2}$, the rank of the above matrix is maximum possible and is equal to $d-1$. This means that for every choice of polynomial $G_{d-1}$ in this case we have 
$$
h_{A^{(1)}}(\frac{d}{2}-1)=d-1.
$$
In order to find possbile values for $h_A(\frac{d}{2})$, note that the rank of $\Cat_F(\frac{d}{2})$ is at most equal to $\frac{3d}{2}$. Also
$$\frac{3d}{2}-2=\rk\Cat_{G_{d-2}}(\frac{d}{2}-2)+\rk\Cat_{G_{d-2}}(\frac{d}{2}-1)+\rk\Cat_{G_{d-2}}(\frac{d}{2})\leq \rk\Cat_F(\frac{d}{2})\leq \frac{3d}{2}.$$ 
Note that setting $G_{d-2}$ as  (\ref{Gd-2maxeven}), $G_{d-1}=0$ and  $G_d$ equal to the following 
\begin{equation}\label{Gdmaxeven}
G_{d}=\left\{
                \begin{array}{ll}
                  0 & \text{for $t=\frac{3d}{2}-2$},\\
                  \frac{Y^{d}}{(d)!} & \text{for $t=\frac{3d}{2}-1,$}\\
                  \frac{Y^{d}}{(d)!}+ \frac{Z^{d}}{(d)!}& \text{for $t=\frac{3d}{2}$}.\\
                \end{array}
              \right.
\end{equation}
provides the desired ranks for the catalecticant matrix $\Cat_F(\frac{d}{2})$  in (\ref{catmatrix3lines}). 
\end{proof}
We now prove that the rank matrix of $A$, or equivalently,  Hilbert functions of $A$, $A^{(1)}$ and $A^{(2)}$  are  completely determined by the maximum values of  $h_{A^{(2)}}$, $h_{A^{(1)}}$ and $h_{A}$. We then provide all rank matrices for each possible combination of integers  $(r,s,t)$ listed in Lemma \ref{maxvaluesevenLemma}. 
\begin{theorem}[Even socle degree]\label{3linesHFtheorem-even}
Let $A$ be an Artinian Gorenstein algebra with even socle degree $d\geq 2$ and $\ell\in A_1$ such that $\ell^2\neq 0$ and  $\ell^3=0$. Then Hilbert functions of $A$, $A^{(1)}$ and $A^{(2)}$  are  completely determined by $(r,s,t)=(h_{A^{(2)}}(\frac{d}{2}-1) ,h_{A^{(1)}}(\frac{d}{2}-1),h_A(\frac{d}{2}))$. More precisely, 
\begin{itemize}
\item[$(1)$] if $d\geq 4$, $r\in [1,\frac{d}{2}-1]$, $s\in[2r,\frac{d}{2}+r]$ and $t\in [2s-r,\frac{d}{2}+s+1]$, then  
\begin{equation}\label{HFeven(1)}
h_{A^{(2)}}(i)=\left\{
                \begin{array}{ll}
                  i+1 &  0\leq i\leq r-1,\\
                  r & r\leq i\leq \frac{d}{2}-1,\\
                \end{array}
              \right.\quad
h_{A^{(1)}}(i)=\left\{
                \begin{array}{ll}
                  2i+1 &  0\leq i\leq r-1,\\
                  i+r+1 & r\leq i\leq s-r-1,\\
                  s & s-r\leq i\leq \frac{d}{2}-1.
                \end{array}
              \right.
 \end{equation}
  \begin{itemize}
\item If $t=3r$ then there are two possible Hilbert functions for $A$
\begin{equation}\label{HFevenA(1,1)}
h_{A}(i)=\left\{
                \begin{array}{ll}
                 1 & i=0,\\
                  3i &  1\leq i\leq r-1,\\
                  3r & r\leq i\leq\frac{d}{2},\\
                \end{array}
              \right.
\text{and} \quad 
h_{A}(i)=\left\{
                \begin{array}{ll}
                   1 & i=0,\\
                  3i &  1\leq i\leq r-1,\\
                  3r-1& i=r,\\
                  3r & r+1\leq i\leq \frac{d}{2},\\
                \end{array}
              \right.
 \end{equation}

\item otherwise, i.e., $t>3r$ we have
\begin{equation}\label{HFevenA(1)o.w.}
h_{A}(i)=\left\{
                \begin{array}{ll}
                   1 & i=0,\\
                  3i &  1\leq i\leq r,\\
                 2i+r+1 & r+1\leq i\leq s-r-1,\\
                 2i+r+1& i=s-r, \hspace*{2mm}\text{if}\hspace*{2mm}t>2s-r\hspace*{2mm}\text{and}\hspace*{2mm}s>2r,\\
                  2i+r& i=s-r, \hspace*{2mm}\text{if}\hspace*{2mm}t>2s-r \hspace*{2mm}\text{and}\hspace*{2mm}s=2r,\\
                 i+s+1 &s-r+1\leq i\leq t-s-1,\\
                 t & t-s \leq i\leq \frac{d}{2}.
                \end{array}
              \right.
 \end{equation}
 \end{itemize}
 \item[$(2)$]
If $d\geq 2$, $r=\frac{d}{2}$, $s=d-1$ and $t\in[\frac{3d}{2}-2,\frac{3d}{2}]$, 
then for every  $0\leq i\leq \frac{d}{2}-1$
 \begin{align}\label{HFeven(2)}
  h_{A^{(2)}}(i)=i+1,\quad  h_{A^{(1)}}(i)=2i+1, \hspace*{2mm}\text{and}
 \end{align}
  \begin{align}\label{HFevenA(2.1)}
 h_{A}(i)=\left\{
                \begin{array}{ll}
                 1 & i=0,\\
                  3i &  1\leq i\leq \frac{d}{2}-1,\\
                   t &  i=\frac{d}{2}.\\
                \end{array}
              \right.
  \end{align}
    \end{itemize}
\end{theorem}

\begin{proof}
We first show $(1)$. Since the Hilbert function of Artinian    Gorenstein algebras are symmetric it is enough to determine it up to the middle degree.
We have that 
$$
A^{(2)} = S/\ann(\ell^2\circ F) = S/\ann(G_{d-2}).
$$
So $A^{(2)}$ is an Artinian    Gorenstein algebra with codimension at most two and the maximum value of $h_{A^{(2)}}$ is equal to $r$. The Hilbert function of $A^{(2)}$ increases by exactly one until it reaches $r$ and it stays $r$ up to the middle degree, $\frac{d}{2}-1$. So we get $h_{A^{(2)}}$ as we claimed.\par 
The assumption on $r$ implies that $h_{A^{(2)}}(\frac{d}{2}-2)=h_{A^{(2)}}(\frac{d}{2}-1)=r$. So 
$$
(h_{A^{(1)}}-(h_{A^{(2)}})_+)(\frac{d}{2}-1) = h_{A^{(1)}}(\frac{d}{2}-1)-h_{A^{(2)}}(\frac{d}{2}-2) = s-r.
$$
Since $(h_{A^{(1)}}-(h_{A^{(2)}})_+)(1)\leq 2$, Lemma \ref{diffO-seq} implies that for every $0\leq i\leq s-r-1,$
$$(h_{A^{(1)}}-(h_{A^{(2)}})_+)(i)=i+1.$$
So since $0\leq r-1\leq s-r-1$, for every $0\leq i\leq r-1$ we have that 
$$
h_{A^{(1)}}(i) = i+1+h_{A^{(2)}}(i-1)= i+1+i = 2i+1.
$$
If $r-1< s-r-1$, then for $r\leq i\leq s-r-1$ we have 
$$
h_{A^{(1)}}(i) = i+1+h_{A^{(2)}}(i-1)= i+1+r.
$$
We have that $r\leq s-r$, so  $h_{A^{(2)}}(i)=r$ for every $s-r-1\leq i\leq \frac{d}{2}-1$ which implies that  $h_{A^{(1)}}(i) = s$, for every $s-r\leq i\leq \frac{d}{2}-1$.

We now determine the Hilbert function of $A$. By assumption we have $(h_A-(h_{A^{(1)}})_+)({\frac{d}{2}})= h_{A}(\frac{d}{2})-h_{A^{(1)}}(\frac{d}{2}-1)=t-s$. On the other hand, $(h_{A}-(h_{A^{(1)}})_+)(1)\leq 2$ and by Lemma \ref{diffO-seq} we conclude that $h_{A}-(h_{A^{(1)}})_+$ is the Hilbert function of some algebra with codimension at most two. So for every $0\leq i\leq t-s-1$
$$
(h_A-(h_{A^{(1)}})_+)({i}) = i+1.
$$
By assumption we have $0\leq r-1\leq s-r-1\leq t-s-1$, so for every $1\leq i\leq r-1$
$$
h_A(i) =i+1+h_{A^{(1)}}(i-1)= i+1+2(i-1)+1=3i.
$$
\begin{itemize}
\item Suppose that $r=t-s$, then $s=2r$ and $t=3r$. Since we have $r\leq \frac{d}{2}-1$ and the Hilbert function of an  algebra with codimension two is unimodal, we get
$$
\left(h_A-(h_{A^{(1)}})_+\right)({r})\geq \left(h_A-(h_{A^{(1)}})_+\right)({r-1})
$$ and thus 
$$h_A(r)\geq r+h_{A^{(1)}}(r-1) = r+2(r-1)+1=3r-1.
$$
Thus we have two possible values for $h_A(r)$, that is either equal to $3r-1$ or $3r$. Clearly, $h_A(i)=3r$ for every $r+1\leq i\leq \frac{d}{2}$. 
\item Now suppose that $r<t-s$. Then $$h_A(r) = r+1+h_{A^{(1)}}(r-1)=r+1+2(r-1)+1=3r.$$
If $r<s-r-1$, then for every $r+1\leq i\leq s-r-1$ we get 
$$
h_A(i)=i+1+h_{A^{(1)}}(i-1) = i+1+r+i=2i+r+1.
$$
If $s-r-1<t-s-1$, then 
\begin{equation*}
h_A(s-r) =s-r+1+h_{A^{(1)}}(s-r-1) = \left\{
                \begin{array}{ll}
                  s-r+1+s &  \text{if}\hspace*{2mm} s>2r,\\
                  s-r+1+s-1 & \text{if}\hspace*{2mm} s=2r.\\
                \end{array}
              \right.
\end{equation*}
If $s-r<t-s-1$, then for every $s-r+1\leq i\leq t-s-1$ we get 
$$
h_A(i)=i+1+h_{A^{(1)}}(i-1)=i+1+s.
$$
Since the Hilbert function of an Artinian algebra with codimension two is unimodal and $t-s\leq \frac{d}{2}+1$ we have that
$$
\left(h_A-(h_{A^{(1)}})_+\right)({t-s})\geq \left(h_A-(h_{A^{(1)}})_+\right)({t-s-1})=t-s.
$$
Therefore, 
\begin{equation}\label{h(t-s)}
h_A(t-s)\geq t-s+h_{A^{(1)}}(t-s-1).
\end{equation}
If $s-r-1<t-s-1$, then $h_A(t-s) \geq t-s+s = t$. Therefore, for every $t-s\leq i\leq \frac{d}{2}$ we have that $h_A(i)=t$.

\noindent If $s-r=t-s$, assuming  $r= s-r$ implies that $s=2r$ and $t=3r$ which contradicts the assumption that $r<t-s$. So we have $r\leq s-r-1$. Using (\ref{h(t-s)}) we get that
$$
h_A(t-s)\geq t-s+h_{A^{(1)}}(t-s-1)= t-s+s=t .
$$ We conclude that $h_A(i)=t$, for every $t-s\leq i\leq \frac{d}{2}$. 
\end{itemize}
We now prove $(2)$. Notice that 
$$\frac{d}{2}-1=\frac{3d}{2}-2 -(d-1)\leq \left(h_A-(h_{A^{(1)}})_+\right)({\frac{d}{2}})\leq \frac{3d}{2} -(d-1)=\frac{d}{2}+1.$$
If $ \frac{d}{2}\leq \left(h_A-(h_{A^{(1)}})_+\right)({\frac{d}{2}})\leq \frac{d}{2}+1$, then for every $1\leq i\leq \frac{d}{2}-1$ we have that $\left(h_A-(h_{A^{(1)}})_+\right)(i)=i+1$ which implies that
$$
h_A(i)=i+1+2(i-1)+1=3i.
$$
If $\left(h_A-(h_{A^{(1)}})_+\right)({\frac{d}{2}})= \frac{d}{2}-1$, then for $1\leq i\leq \frac{d}{2}-2$ we have $h_A(i)=i+1+2(i-1)+1=3i$. On the other hand, for every $d\geq 6$ we have that 
\begin{equation*}
\Cat_F(\frac{d}{2}-1)=\left[\begin{array}{@{}c|c|c@{}}
\mathbf{0}& \mathbf{0}&{\Cat_{G_{d-2}}{(\frac{d}{2}-3)}}
\\\hline
\mathbf{0}&\Cat_{G_{d-2}}{(\frac{d}{2}-2)}&\Cat_{G_{d-1}}{(\frac{d}{2}-2)}\\\hline
\Cat_{G_{d-2}}{(\frac{d}{2}-1)}&\Cat_{G_{d-1}}{(\frac{d}{2}-1)}&\Cat_{G_{d}}{(\frac{d}{2}-1)}
\end{array}
\right].\\
\end{equation*}
Which implies that 
 $$
 \frac{3d}{2}-3= h_{A^{(2)}}(\frac{d}{2}-3)+h_{A^{(2)}}(\frac{d}{2}-2)+h_{A^{(2)}}(\frac{d}{2}-1)\leq h_A(\frac{d}{2}-1), 
 $$
 and since  $\Cat_F(\frac{d}{2}-1)$ is a square matrix of size $\frac{3d}{2}-3$ we get  $h_A(\frac{d}{2}-1)=\frac{3d}{2}-3$. For $d=4$ similar argument implies that 
  $$
3= h_{A^{(2)}}(0)+h_{A^{(2)}}(1)\leq h_A(1).
 $$
For $d=2$ there is noting to show.
\end{proof}
Now we state and prove the analogues statements to Lemma \ref{maxvaluesevenLemma} and Theorem \ref{3linesHFtheorem-even} for Artinian Gorenstein  algebras with odd socle degrees.

\begin{lemma}[Odd socle degree]\label{maxvaluesoddLemma}
There exists an Artinian Gorenstein algebra $A$ with odd socle degree $d\geq 3$ and linear form $\ell\in A_1$ where $\ell^2\neq 0$ and $\ell^3=0$, such that 
$$
(r,s,t)=\left( h_{A^{(2)}}(\frac{d-1}{2}),  h_{A^{(1)}}(\frac{d-1}{2}),  h_{A}(\frac{d-1}{2})\right)
$$
if and only if 
\begin{itemize}
\item[$(1)$] $r\in [1,\frac{d-1}{2}-1], s\in[2r,\frac{d-1}{2}+r]$ and $t\in[2s-r,\frac{d-1}{2}+s+1]$, for $d\geq 5$; or 
\item[$(2)$] $r\in [1,\frac{d-1}{2}-1], s=\frac{d-1}{2}+r+1$ and $t=d+r$, for $d\geq 5$; or 
\item[$(3)$] $r=\frac{d-1}{2}, s\in [d-1,d]$ and $t\in [\frac{d-1}{2}+s-1,3\frac{d-1}{2}]$, for $d\geq 3$.
\end{itemize}
\end{lemma}
\begin{proof}
The maximum value of the Hilbert function of $A$ occurs in degree $\frac{d-1}{2}$ and it is equal to the rank of the following catalecticant matrix
 \begin{equation}
\Cat_F(\frac{d-1}{2})=\left[\begin{array}{@{}c|c|c@{}}
\mathbf{0}& \mathbf{0}&{\Cat_{G_{d-2}}{(\frac{d-1}{2}-2)}}
\\\hline
\mathbf{0}&\Cat_{G_{d-2}}{(\frac{d-1}{2}-1)}&\Cat_{G_{d-1}}{(\frac{d-1}{2}-1)}\\\hline
\Cat_{G_{d-2}}{(\frac{d-1}{2})}&\Cat_{G_{d-1}}{(\frac{d-1}{2})}&\Cat_{G_{d}}{(\frac{d-1}{2})}
\end{array}
\right]\\
\end{equation}
which is equal to
\begin{Small}
\begin{equation}\label{catmatrix3linesodd}
 \Cat_F(\frac{d-1}{2}) = \left[\begin{array}{@{}cccc|cccc|cccc@{}}
    0 & 0 & \cdots & 0 &  0 & 0 & \cdots & 0 & c_{0} & c_{1}&\cdots & c_{\frac{d+1}{2}} \\
     0 & 0 & \cdots & 0 & 0 & 0 &\cdots & 0 & c_{1} & c_{2} &\cdots & c_{\frac{d+1}{2}+1} \\
    \vdots & \vdots & \reflectbox{$\ddots$}  &\vdots & \vdots & \vdots & \reflectbox{$\ddots$}  &\vdots  & \vdots  &  \vdots & \reflectbox{$\ddots$} &\vdots\\
 0 & 0 & \cdots & 0 &  0 & 0 &\cdots & 0 &  c_{\frac{d-1}{2}-2} & c_{\frac{d-1}{2}-1} &\cdots & c_{d-2} \\\hline
  0 & 0 & \cdots & 0 &  c_{0} & c_{1}&\cdots & c_{\frac{d-1}{2}} & b_{0} & b_{1}&\cdots & b_{\frac{d+1}{2}} \\
 0 & 0 & \cdots & 0 &   c_{1} & c_{2} &\cdots & c_{\frac{d-1}{2}+1}  & b_{1} & b_{2} &\cdots & b_{\frac{d+1}{2}+1} \\
    \vdots  &\vdots & \reflectbox{$\ddots$}  &  \vdots&   \vdots  &\vdots & \reflectbox{$\ddots$}  &  \vdots& \vdots  &  \vdots & \reflectbox{$\ddots$} &\vdots\\
  0 & 0 & \cdots & 0 &  c_{\frac{d-1}{2}-1} & c_{\frac{d-1}{2}} &\cdots & c_{d-2} & b_{\frac{d-1}{2}-1} & b_{\frac{d-1}{2}} &\cdots & b_{d-1} \\\hline
     c_{0} & c_{1}&\cdots & c_{\frac{d-1}{2}-1} & b_{0} & b_{1}&\cdots & b_{\frac{d-1}{2}} &a_{0} & a_{1}&\cdots & a_{\frac{d+1}{2}} \\
 c_{1} & c_{2} &\cdots & c_{\frac{d-1}{2}}  & b_{1} & b_{2} &\cdots & b_{\frac{d-1}{2}+1}& a_{1} & a_{2} &\cdots & a_{\frac{d+1}{2}+1} \\
     \vdots  &\vdots & \reflectbox{$\ddots$}  &  \vdots& \vdots  &  \vdots & \reflectbox{$\ddots$} &\vdots& \vdots  &  \vdots & \reflectbox{$\ddots$} &\vdots\\
    c_{\frac{d-1}{2}} & c_{\frac{d-1}{2}+1} &\cdots & c_{d-2} & b_{\frac{d-1}{2}} & b_{\frac{d-1}{2}+1} &\cdots & b_{d-1}& a_{\frac{d-1}{2}} & a_{\frac{d-1}{2}+1} &\cdots & a_{d} \\
     \end{array}\right].
\end{equation}
\end{Small}
We note that $r=h_{A^{(2)}}(\frac{d-1}{2})\in [1,\frac{d-1}{2}]$. First assume that $r\in [1,\frac{d-1}{2}-1]$ and note that  the socle degree of $A^{(2)}$ is odd, then we have that $
 h_{A^{(2)}}(\frac{d-1}{2}-2)= h_{A^{(2)}}(\frac{d-1}{2}-1)= \linebreak h_{A^{(2)}}(\frac{d-1}{2})=r.
$
 We may assume that the ranks of the lower right submatrices of \linebreak$\Cat_{G_{d-2}}(\frac{d-1}{2}-2), \Cat_{G_{d-2}}(\frac{d-1}{2}-1)$ and $\Cat_{G_{d-2}}(\frac{d-1}{2})$ are equal to $r$. Setting  $c_{d-r-1}=1$ and $c_i=0$ for every $i\neq d-r-1$, or equivalently setting $G_{d-2}$ as the following  provides the desired property
\begin{equation}\label{Gd-2}
G_{d-2}= \frac{Y^{r-1}Z^{d-r-1}}{(r-1)!(d-r-1)!}, \quad \text{for each}\quad r\in[1,\frac{d-1}{2}-1].
\end{equation}
The Hilbert function of $A^{(1)}$ in degree $\frac{d-1}{2}$ is equal to $2r+\rk\mathbf{B}$ where 
\begin{equation}\label{matrixBodd}
\mathbf{B}=\left(\begin{array}{@{}ccccccc@{}}
 b_{0} &\cdots & b_{\frac{d-1}{2}-r} \\
    \vdots  & \reflectbox{$\ddots$} &\vdots\\
b_{\frac{d-1}{2}-r}&\cdots & b_{d-1-2r}      \end{array}\right).
\end{equation}
So $s=h_{A^{(1)}}(\frac{d-1}{2})\in [2r, \frac{d-1}{2}+r+1]$. Suppose that  $s\in [2r, \frac{d-1}{2}+r]$. This implies that $ h_{A^{(1)}}(\frac{d-1}{2}-1)= h_{A^{(1)}}(\frac{d-1}{2}) = s$. To prove $(1)$, we use the same argument that we used to prove Lemma \ref{maxvaluesevenLemma} part $(1)$. Therefore, the following choice of $G_{d-1}$ and $G_d$ completes the proof of part $(1)$.
\begin{equation*}
G_{d-1}=\left\{
                \begin{array}{ll}
                  0 & \text{if $s-2r=0$},\\
                  \frac{Y^{d-s+2r}Z^{s-2r-1}}{(d-s+2r)!(s-2r-1)!} & \text{if $1\leq s-2r\leq \frac{d-1}{2}-r,$}\\
                \end{array}
              \right.
\end{equation*}
and 
\begin{equation*}
G_{d}=\left\{
                \begin{array}{ll}
                  0 & \text{if $t-2s+r=0$},\\
                  \frac{Y^{d-t+3r+1}Z^{t-3r-1}}{(d-t+3r+1)!(t-3r-1)!} & \text{if $1\leq t-2s+r\leq \frac{d-1}{2}+r-s+1.$}\\
                \end{array}
              \right.
\end{equation*}
Now assume that $s=h_{A^{(1)}}(\frac{d-1}{2})=\frac{d-1}{2}+r+1$, which is the maximum possible for $r\in [1,\frac{d-1}{2}-1]$.  The following submatrices  of $\Cat_{G_{d-1}}(\frac{d-1}{2})$ having maximal rank, that is equal to $\frac{d-1}{2}-r+1$, implies that $\rk\Cat_{x\circ F}(\frac{d-1}{2})=\frac{d-1}{2}+r+1$.
\begin{equation*}
\mathbf{B}=\left[\begin{array}{@{}ccccccc@{}}
 b_{0} &\cdots & b_{\frac{d-1}{2}-r} \\
    \vdots  & \reflectbox{$\ddots$} &\vdots\\
b_{\frac{d-1}{2}-r}&\cdots & b_{d-1-2r}      \end{array}\right].
\end{equation*}
This forces the following submatrix of $\Cat_{G_{d-1}}(\frac{d-1}{2}-1)$ to have maximal rank, that is equal to  $\frac{d-1}{2}-r$. 
\begin{equation*}
\mathbf{B^\prime}=\left[\begin{array}{@{}ccccccc@{}}
 b_{0} &\cdots & b_{\frac{d+1}{2}-r} \\
    \vdots  & \reflectbox{$\ddots$} &\vdots\\
b_{\frac{d-1}{2}-1-r}&\cdots & b_{d-1-2r}      \end{array}\right].
\end{equation*}
Setting $b_{\frac{d-1}{2}+r}=1$ and $b_{i}=0$ for every $i\neq \frac{d-1}{2}+r$, or equivalently, 
$$
G_{d-1} = \frac{Y^{\frac{d-1}{2}-r}Z^{\frac{d-1}{2}+r}}{(\frac{d-1}{2}-r)!(\frac{d-1}{2}+r)!}
$$
provides that 
$$h_{A^{(1)}}(\frac{d-1}{2})=\frac{d-1}{2}+r+1, \hspace*{2mm}\text{and} \hspace*{2mm} h_{A^{(1)}}(\frac{d-1}{2}-1)=\frac{d-1}{2}+r.$$
Since $\mathbf{B}$ and $\mathbf{B^\prime}$ both have maximal ranks for every choice of $G_d$, we conclude 
$$
h_A(\frac{d-1}{2})=\rk\Cat_{F}(\frac{d-1}{2}) = 3r+\rk\mathbf{B}+\rk\mathbf{B^\prime} = d+r.
$$
Now we assume that $r=\frac{d-1}{2}$ as in $(3)$. Since $d$ is an odd integer $h_{A^{(2)}}(\frac{d-1}{2}-1)=h_{A^{(2)}}(\frac{d-1}{2})=\frac{d-1}{2}$ and  $h_{A^{(2)}}(\frac{d-1}{2}-2) = \frac{d-1}{2}-1$. Setting $c_{\frac{d-1}{2}-1}=1$ and $c_i=0$ for every $i\neq \frac{d-1}{2}-1$, that is
\begin{equation}\label{random}
G_{d-2} = \frac{Y^{\frac{d-1}{2}}Z^{\frac{d-1}{2}-1}}{(\frac{d-1}{2})!(\frac{d-1}{2}-1)!},
\end{equation}
implies that $h_{A^{(2)}}(\frac{d-1}{2})=\frac{d-1}{2}$.
In order to find possible values for $h_{A^{(1)}}(\frac{d-1}{2})$, note that
\begin{equation*}
h_{A^{(1)}}(\frac{d-1}{2}) = \rk\left[\begin{array}{@{}c|c@{}}
 \mathbf{0}&{\Cat_{G_{d-2}}{(\frac{d-1}{2}-1)}}
\\\hline
\Cat_{G_{d-2}}{(\frac{d-1}{2})}&\Cat_{G_{d-1}}{(\frac{d-1}{2})}\\
\end{array}
\right],\\
\end{equation*}
is a square matrix of size $d$. On the other hand $$d-1=\rk\Cat_{G_{d-2}}(\frac{d-1}{2}-1)+\rk\Cat_{G_{d-2}}(\frac{d-1}{2}-1)\leq h_{A^{(1)}}(\frac{d-1}{2}).$$ 
For the polynomial $G_{d-2}$ as in (\ref{random}) we get that the last column of the above matrix is zero. So setting $G_{d-1}=0$ gives $h_{A^{(1)}}(\frac{d-1}{2})=d-1$ and setting $G_{d-1}=\frac{Z^{d-1}}{(d-1)!}$ gives that $h_{A^{(1)}}(\frac{d-1}{2})=d$.
To find possible values for $h_A(\frac{d-1}{2})$ we note that the number of rows in the catalecticant matrix (\ref{catmatrix3linesodd}) is equal to  $3\frac{d-1}{2}$. 
If $h_{A^{(1)}}(\frac{d-1}{2})=d$, then   independent of  the choice of $G_d$,the Hilbert function $h_A(\frac{d-1}{2})$ is equal to the maximum possible. So,
$$\rk\Cat_F(\frac{d-1}{2})\geq  h_{A^{(1)}}(\frac{d-1}{2})+h_{A^{(2)}}(\frac{d-1}{2}-2)=d+\frac{d-1}{2}-1=3\frac{d-1}{2}.$$
If $h_{A^{(1)}}(\frac{d-1}{2})=d-1$, then 
$$
\rk\Cat_F(\frac{d-1}{2})\geq  h_{A^{(1)}}(\frac{d-1}{2})+h_{A^{(2)}}(\frac{d-1}{2}-2)=d-1+\frac{d-1}{2}-1=3\frac{d-1}{2}-1.
$$ Setting $G_{d-1}=0$ and $G_{d}=0$ provides that 
$ h_{A}(\frac{d-1}{2})=3\frac{d-1}{2}-1$. And setting $G_{d-1}=0$ and $G_{d}=\frac{Z^d}{d!}$ provides that $ h_{A}(\frac{d-1}{2})=3\frac{d-1}{2}$. In fact, with this choice  the last column of  $\Cat_F(\frac{d-1}{2})$ becomes non-zero and linearly independent from the previous columns.
\end{proof}
%In the following theorem we determine the rank matrix of an Artinian Gorenstein algebra with the associated triple $(r,s,t)$ provided in the above lemma.
\begin{theorem}[Odd socle degree]\label{3linesHFtheorem-odd}
Let $A$ be an Artinian Gorenstein algebra with odd socle degree $d\geq 3$ and $\ell\in A_1$ such that $\ell^2\neq 0$ and $\ell^3=0$. Then Hilbert functions of $A$, $A^{(1)}$ and $A^{(2)}$ are completely determined by $(r,s,t)=( h_{A^{(2)}}(\frac{d-1}{2}),  h_{A^{(1)}}(\frac{d-1}{2}),  h_{A}(\frac{d-1}{2}))$. More precisely,   
\begin{itemize}
\item[$(1)$] if $d\geq 5$, $r\in [1,\frac{d-1}{2}-1], s\in[2r,\frac{d-1}{2}+r]$ and $t\in [2s-r,\frac{d-1}{2}+s+1]$, then 
\begin{equation}\label{HFodd(1)}
h_{A^{(2)}}(i)=\left\{
                \begin{array}{ll}
                  i+1 & 0\leq i\leq r-1,\\
                  r & r\leq i\leq\frac{d-1}{2},\\
                \end{array}
              \right.\quad
h_{A^{(1)}}(i)=\left\{
                \begin{array}{ll}
                  2i+1 &  0\leq i\leq r-1,\\
                   i+r+1 &  r\leq i\leq s-r-1,\\
                  s & s-r\leq i\leq \frac{d-1}{2}.\\
                \end{array}
              \right.
 \end{equation}
 \begin{itemize}
\item If $t=3r$ then there are two possible Hilbert functions for $A$
\begin{equation}\label{HFoddA(1,1)}
h_{A}(i)=\left\{
                \begin{array}{ll}
                 1 & i=0,\\
                  3i &  1\leq i\leq r-1,\\
                  3r & r\leq i\leq\frac{d-1}{2}.\\
                \end{array}
              \right.
\text{and} \quad 
h_{A}(i)=\left\{
                \begin{array}{ll}
                   1 & i=0,\\
                  3i &  1\leq i\leq r-1,\\
                  3r-1& i=r,\\
                  3r & r+1\leq i\leq \frac{d-1}{2},\\
                \end{array}
              \right.
 \end{equation}
\item otherwise
\begin{equation}\label{HFoddA(1,2)}
h_A(i)= \left\{
                \begin{array}{ll}
                   1 & i=0,\\
                  3i &  1\leq i\leq r,\\
                  2i+r+1& r+1\leq i\leq s-r-1,\\
                  2i+r+1& i=s-r, \hspace*{2mm}\text{if}\hspace*{2mm}t>2s-r\hspace*{2mm}\text{and}\hspace*{2mm}s>2r,\\
                  2i+r& i=s-r, \hspace*{2mm}\text{if}\hspace*{2mm}t>2s-r \hspace*{2mm}\text{and}\hspace*{2mm}s=2r,\\
                  i+s+1 & s-r+1\leq i\leq t-s-1,\\
                  t & t-s\leq i\leq \frac{d-1}{2}.\\
                \end{array}
              \right.
\end{equation}
 \end{itemize}
 \item[$(2)$] If $d\geq 3$, $r\in[1,\frac{d-1}{2}-1],s=\frac{d-1}{2}+r+1 $ and $t=d+r$, then 
\begin{equation}\label{HFodd(2)}
h_{A^{(1)}}(i)= \left\{
                \begin{array}{ll}
                  2i+1 &  0\leq i\leq r,\\
                  i+1+r & r+1\leq i\leq\frac{d-1}{2},\\
                \end{array}
              \right. \text{and}\quad h_{A}(i)= \left\{
                \begin{array}{ll}
                1& i=0,\\
                  3i &  1\leq i\leq r+1,\\
                  2i+1+r & r+2\leq i\leq\frac{d-1}{2}.\\
                \end{array}
              \right. 
\end{equation}
\item[$(3)$] For $d\geq 3$ if  $r=\frac{d-1}{2}, s\in [d-1,d]$ and $t\in [\frac{d-1}{2}+s-1,3\frac{d-1}{2}]$, then for every $i\in [0,\frac{d-1}{2}-1]$
  \begin{equation}\label{HFodd(3)}
 h_{A^{(2)}}(i) = i+1, \quad h_{A^{(1)}}(i) = 2i+1\hspace*{2mm}\text{and} \hspace*{2mm} h_A(0)=1,  h_A(i) =3i.
  \end{equation}
 \end{itemize}
\end{theorem}
\begin{proof}
We first  prove part  $(1)$. The assumptions on $r$ and $s$ imply that 
$$ h_{A^{(1)}}(\frac{d-1}{2}-1)=h_{A^{(1)}}(\frac{d-1}{2})=h_{A^{(1)}}(\frac{d-1}{2}+1)=s,
$$
and 
$$ h_{A^{(2)}}(\frac{d-1}{2}-1)=h_{A^{(2)}}(\frac{d-1}{2})=r.
$$
Therefore, applying Theorem \ref{3linesHFtheorem-even} part $(1)$ for $d-1$ completes the proof of $(1)$.\par 
\noindent Now we show $(2)$. First note that $h_{A^{(2)}}$ is the same as the previous case. By the assumption we have that 
\begin{align*}
(h_{A^{(1)}}-(h_{A^{(2)}})_+)(\frac{d-1}{2}) &= h_{A^{(1)}}(\frac{d-1}{2})-h_{A^{(2)}}(\frac{d-1}{2}-1)\\
&= h_{A^{(1)}}(\frac{d-1}{2})-h_{A^{(2)}}(\frac{d-1}{2})\\
&=\frac{d-1}{2}+r+1-r= \frac{d-1}{2}+1.
\end{align*}
Using Lemma \ref{diffO-seq}, we get that
$(h_{A^{(1)}}-(h_{A^{(2)}})_+)(i) = i+1
$  for every $i\in[0,\frac{d-1}{2}]$.
Therefore, $h_{A^{(1)}}$ is what we claimed.
To obtain $h_A$, we note that 
\begin{align*}
(h_{A}-(h_{A^{(1)}})_+)(\frac{d-1}{2})  = h_{A}(\frac{d-1}{2})-h_{A^{(1)}}(\frac{d-1}{2}-1)= d+r- h_{A^{(1)}}(\frac{d-1}{2}-1).
\end{align*}
If $r<\frac{d-1}{2}-1$, then we have $h_{A^{(1)}}(\frac{d-1}{2}-1)=\frac{d-1}{2}+r$ and if $r=\frac{d-1}{2}-1$, then $h_{A^{(1)}}(\frac{d-1}{2}-1)=2(\frac{d-1}{2}-1)+1=d-2.$ In both cases we get that 
\begin{align*}
(h_{A}-(h_{A^{(1)}})_+)(\frac{d-1}{2})  = d+r-(d-2) = r+2 = \frac{d-1}{2}+1.
\end{align*}
So for every $i\in[0,\frac{d-1}{2}]$ we have $(h_{A}-(h_{A^{(1)}})_+)(i) = i+1,
$ and therefore
$
h_A(i)=i+1+h_{A^{(1)}}(i-1)
$
which implies the desired Hilbert function for $A$.

To prove $(3)$ we get $h_{A^{(2)}}$ by replacing $r$ by $\frac{d-1}{2}$ in the previous case. By the assumption we have that 
$$
\frac{d-1}{2}\leq (h_{A^{(1)}}-(h_{A^{(2)}})_+)(\frac{d-1}{2})\leq \frac{d-1}{2}+1.
$$
So for every $i\in [0,\frac{d-1}{2}-1]$  $$h_{A^{(1)}}(i)=i+1+h_{A^{(2)}}(i-1)=2i+1.$$
To obtain $h_A$ we observe that
\begin{align*}
(h_{A}-(h_{A^{(1)}})_+)(\frac{d-1}{2}) \geq \frac{d-1}{2}+s-1-(d-2)
=s-\frac{d-1}{2}\geq \frac{d-1}{2}.
\end{align*}
Therefore,($h_{A}-(h_{A^{(1)}})_+)(i)=i+1$ for every $i\in [0,\frac{d-1}{2}-1]$, and equivalently, we have $h_A(0)=1$ and $h_A(i)=i+1+2(i-1)+1=3i$ for every $i\in [1,\frac{d-1}{2}-1]$.
\end{proof}
We prove that the lists of rank matrices given in Theorems \ref{3linesHFtheorem-even} and \ref{3linesHFtheorem-odd} are exhaustive lists. 
\begin{theorem}\label{allHF´sPossible-even}
A vector of non-negative integers $h$ is the Hilbert function of some Artinian    Gorenstein algebra  $A=S/\ann(F)$ such that there exists a linear form $\ell$ satisfying $ \ell^2\neq 0$ and $\ell^3=0$ if and only if $h$ is equal to one of the Hilbert functions provided in Theorems \ref{3linesHFtheorem-even} and \ref{3linesHFtheorem-odd}.
\end{theorem}
\begin{proof}
In Lemmas \ref{maxvaluesevenLemma} and \ref{maxvaluesoddLemma} we provide the complete list of possible values for the maximum of the Hilbert function of any Artinian Gorenstein algebras $A$ where $\ell^3=0$. In fact, for each maximum value we produce a dual generator $F$ for $A$. On the other hand, in Theorems \ref{3linesHFtheorem-even} and \ref{maxvaluesoddLemma}  we prove that for each possible maximum value the Hilbert function of $A$ is uniquely determined by the maximum value in all the cases except when $r\in [1,\lfloor\frac{d}{2}\rfloor-1]$ and $t=3r$ for every $d\geq 4$, in which we have two possibilities for $h_A$. We show that both Hilbert functions provided for $h_A$ occur for some Artinian Gorenstein algebra $A$.\par 
First assume that $d\geq 6$, $r\in[2,\lfloor\frac{d}{2}\rfloor-1]$ and $t=3r$.  This  implies that $s=2r$.  Fixing $h_A(r)$ to be either $3r-1$ or $3r$ we provide a degree $d$ polynomial satisfying (\ref{F}) as the dual generator for Artinian Gorenstein algebra $A$ such that $h_A(\lfloor\frac{d}{2}\rfloor)=3r$.  
%We provide two polynomials of degree $d$ as dual generators for Artinian Gorenstein algebras satisfying $h_A(\lfloor\frac{d}{2}\rfloor)=3r$  and the Hilbert function of each of which in degree $r$ is equal to $3r-1$ and $3r$ are both possible values for the Hilbert function in degree $r$ of some  Artinian Gorenstein algebra. Fixing $h_A(r)$ to be either $3r-1$ or $3r$ we provide a degree $d$ polynomial as the dual generator for Artinian Gorenstein algebra $A$.  
Pick the following monomial basis for $A_r$
$$\B_r=\{x^r,x^{r-1}y,x^{r-1}z,x^{r-2}y^2,\dots ,y^r,y^{r-1},z,\dots , z^r\}$$
so
\begin{equation}
\Cat_F(r)=\left[\begin{array}{@{}c|c|c@{}}
\mathbf{0}& \mathbf{0}&{\Cat_{G_{d-2}}{(r-2)}}
\\\hline
\mathbf{0}&\Cat_{G_{d-2}}{(r-1)}&\Cat_{G_{d-1}}{(r-1)}\\\hline
\Cat_{G_{d-2}}{(r)}&\Cat_{G_{d-1}}{(r)}&\Cat_{G_{d}}{(r)}
\end{array}
\right].\\
\end{equation}
This is equal to 
\begin{equation}\label{catmatrix3lines-r}
 \Cat_F(r) = \left[\begin{array}{@{}cccc|cccc|cccc@{}}
    0 & 0 & \cdots & 0 &  0 & 0 & \cdots & 0 & c_{0} & c_{1}&\cdots & c_{d-r} \\
     0 & 0 & \cdots & 0 & 0 & 0 &\cdots & 0 & c_{1} & c_{2} &\cdots & c_{d-r+1} \\
    \vdots & \vdots & \reflectbox{$\ddots$}  &\vdots & \vdots & \vdots & \reflectbox{$\ddots$}  &\vdots  & \vdots  &  \vdots & \reflectbox{$\ddots$} &\vdots\\
 0 & 0 & \cdots & 0 &  0 & 0 &\cdots & 0 &  c_{r-2} & c_{r-1} &\cdots & c_{d-2} \\\hline
  0 & 0 & \cdots & 0 &  c_{0} & c_{1}&\cdots & c_{d-r-1} & b_{0} & b_{1}&\cdots & b_{d-r} \\
 0 & 0 & \cdots & 0 &   c_{1} & c_{2} &\cdots & c_{d-r}  & b_{1} & b_{2} &\cdots & b_{d-+1} \\
    \vdots  &\vdots & \reflectbox{$\ddots$}  &  \vdots&   \vdots  &\vdots & \reflectbox{$\ddots$}  &  \vdots& \vdots  &  \vdots & \reflectbox{$\ddots$} &\vdots\\
  0 & 0 & \cdots & 0 &  c_{r-1} & c_{r} &\cdots & c_{d-2} & b_{r-1} & b_{r} &\cdots & b_{d-1} \\\hline
     c_{0} & c_{1}&\cdots & c_{d-r-2} & b_{0} & b_{1}&\cdots & b_{d-r-1} &a_{0} & a_{1}&\cdots & a_{d-r} \\
 c_{1} & c_{2} &\cdots & c_{d-r-1}  & b_{1} & b_{2} &\cdots & b_{d-r}& a_{1} & a_{2} &\cdots & a_{d-r+1} \\
     \vdots  &\vdots & \reflectbox{$\ddots$}  &  \vdots& \vdots  &  \vdots & \reflectbox{$\ddots$} &\vdots& \vdots  &  \vdots & \reflectbox{$\ddots$} &\vdots\\
    c_{r} & c_{r+1} &\cdots & c_{d-2} & b_{r} & b_{r+1} &\cdots & b_{d-1}& a_{r} & a_{r+1} &\cdots & a_{d} \\
     \end{array}\right].
\end{equation}
Using what we have shown in Lemmas \ref{maxvaluesevenLemma} and \ref{maxvaluesoddLemma} part $(1)$, setting $c_{d-r-1}=1$ and all other coefficients in the polynomial $F$ to be zero, or equivalently, 
\begin{equation}\label{3rFeven}
F = \frac{X^2 Y^{r-1}Z^{d-r-1}}{2(r-1)!(d-r-1)!}
\end{equation}
provides that $h_A(\lfloor\frac{d}{2}\rfloor)=3r$.
Therefore, since $G_{d-1}=G_d=0$ we get that 
$$
h_A(r)=\rk\Cat_F(r)=\rk\Cat_{F}(r-2)+ \rk\Cat_{F}(r-1)+\rk\Cat_{F}(r)=3r-1,
$$
where the ranks of $\Cat_{F}(r-2), \Cat_{F}(r-1)$ and $\Cat_{F}(r)$, or equally, $h_{A^{(2)}}(r-2),h_{A^{(2)}}(r-1)$ and $h_{A^{(2)}}(r)$ are given in Theorems \ref{3linesHFtheorem-even} and \ref{3linesHFtheorem-odd}.

In order to provide a polynomial $F$ as the dual generator of $A$ where $h_A(\lfloor\frac{d}{2}\rfloor)=3r=h_A(r)=3r$, we set $c_{d-r-1}=a_{d-r}=1$ and all other coefficients to be zero, so
\begin{equation}\label{3rFeven}
F = \frac{X^2 Y^{r-1}Z^{d-r-1}}{2(r-1)!(d-r-1)!} + \frac{Y^rZ^{d-r}}{r!(d-r)!}.
\end{equation}
We observe that setting $a_{d-r}=1$ in the matrices (\ref{catmatrix3lines}) and  (\ref{catmatrix3linesodd}), the number of linearly independent columns does not increase and is equal to $3r$. On the other hand, setting $a_{d-r}=1$ increases the number of linearly independent columns of $\Cat_F(r)$ in (\ref{catmatrix3lines-r}) by one. In fact, by setting $a_{d-r}=1$ the last column of (\ref{catmatrix3lines-r}) becomes non-zero and not included in the span of the  previous columns. Thus, the number of linearly independent columns in (\ref{catmatrix3lines-r}) is equal to $3r$,  so $h_A(r)=3r$.

\noindent Now assume that $d\geq 4$ and $r=1$. We use the same argument as the previous case for the following matrix.
\begin{equation}\label{catmatrix3lines(r=1)}
\Cat_F(r)=\left[\begin{array}{@{}c|c|c@{}}
 \mathbf{0}&{\Cat_{G_{d-2}}{(r-1)}}&{\Cat_{G_{d-1}}{(r-1)}}
\\\hline
\Cat_{G_{d-2}}{(r)}&\Cat_{G_{d-1}}{(r)}&\Cat_{G_{d}}{(r)}\\
\end{array},
\right]\\
\end{equation}
similarly setting $F = \frac{X^2 Z^{d-2}}{2(d-2)!}$  provides that $\rk\Cat_F(1)=2$, and setting $F = \frac{X^2 Z^{d-2}}{2(d-2)!} + \frac{Y Z^{d-1}}{(d-1)!}$ provides $\rk\Cat_F(1)=3$. We notice that in both cases $h_A(\lfloor\frac{d}{2}\rfloor)=3$.
\end{proof}
As an immediate consequence of above results we get the complete list of possible rank matrices for Artinian Gorenstein algebras and linear forms such that $\ell^2=0$. We may assume that $\ell\neq 0$, since otherwise multiplication map  by $\ell$ is trivial. So we have $A^{(1)}\neq 0$ and $A^{(i)}=0$, for all $i\geq 2$. We denote by $r$ and $s$ the maximum values for $h_{A^{(1)}}$ and $h_A$ respectively.
\begin{corollary}\label{2linescorollary}
There exists an Artinian Gorenstein algebra $A$ with socle degree $d\geq 2$ and $\ell\in A_1$ where $\ell\neq 0$ and $\ell^2=0$,  such that 
$$
(r,s)=\left(h_{A^{(1)}}(\lfloor\frac{d}{2}\rfloor), h_{A}(\lfloor\frac{d}{2}\rfloor)\right)
$$
 if and only if 
 \begin{itemize}
 \item $r\in[1,\lceil\frac{d}{2}\rceil-1]$ and $s\in [2r,\lceil\frac{d}{2}\rceil+r ]$, for $d\geq 3$; or
 \item $r= \lceil\frac{d}{2}\rceil$ and $s=d$ if $d\geq 3$ is odd; $s=d,d+1$ if  $d\geq 2$ is even.
 \end{itemize}
 Moreover, the Hilbert functions of $A$ and $A^{(1)}$ are completely determined by $(r,s)$ as the following 
\begin{equation}\label{HF2lines}
h_{A^{(1)}}(i)=\left\{
                \begin{array}{ll}
                  i+1 & 0\leq i\leq r-1,\\
                  r & r\leq i\leq\lfloor\frac{d}{2}\rfloor.\\
                \end{array}
              \right.\quad
h_{A}(i)=\left\{
                \begin{array}{ll}
                  2i+1 &  0\leq i\leq r-1,\\
                   i+r+1 &  r\leq i\leq s-r-1,\\
                  s & s-r\leq i\leq \lfloor\frac{d}{2}\rfloor.\\
                \end{array}
              \right.
 \end{equation}
\end{corollary}
\begin{proof}
The proof is immediate by considering rank matrices with two non-zero diagonals given by $h_{A^{(1)}}$ and $h_{A^{(2)}}$ provided in Theorems \ref{3linesHFtheorem-even} and  \ref{3linesHFtheorem-odd}.
\end{proof}
\begin{remark}
The above threorems provide complete lists of rank matrices, $M_{\ell,A}$, for Artinian Gorenstein algebras $A$ of codimension two and three satisfying $\ell^3=0$. In fact, there might exists a linear form $\ell^\prime\neq \ell$ such that $\ell^\prime= 0$. 
\end{remark}
We are now able to formulate our last result which provides a formula to compute Jordan types of Artinian Gorenstein algebras with parts of length at most four in terms of at most three parameters $(r,s,t)$ in the above theorems.
Using Remark \ref{r_ij-remark}, we provide the formulas in terms of the ranks of mixed Hessians in certain degrees.
\begin{theorem}\label{JT-theorem}
Let $A=S/\ann(F)$ be an Artinian  Gorenstein algebra with socle degree $d\geq 2$ and  $\ell\neq 0$ be a linear form such that $\ell^4=0$. The Jordan type $P_{\ell,A}$ is one of the followings.
\begin{itemize}
\item If $\ell^3\neq 0$ then the Jordan type partition of $A$ for $\ell$ is given by 
\begin{equation}\label{JT4}
P_{\ell,A}=(\underbrace{4,\dots ,4}_{\Delta^2\mathbf{d}(3)},\underbrace{3,\dots ,3}_{\Delta^2\mathbf{d}(2)}, \underbrace{2,\dots ,2}_{\Delta^2\mathbf{d}(1)},\underbrace{1,\dots ,1}_{\Delta^2\mathbf{d}(0)}),
\end{equation}
where $\mathbf{d}=(\dim_\mathsf{k} A,\dim_\mathsf{k} A^{(1)},\dim_\mathsf{k} A^{(2)}, \dim_\mathsf{k} A^{(3)})$ and the Hilbert functions of $A^{(1)}$, $A^{(2)}$ and $A^{(3)}$ are given in Theorems  \ref{3linesHFtheorem-even} and  \ref{3linesHFtheorem-odd} for parameters 
$$(r,s,t)=\left(\rk\Hess_\ell^{(\lfloor\frac{d}{2}\rfloor-1, \lceil\frac{d}{2}\rceil-2)}(F), \rk\Hess_\ell^{(\lfloor\frac{d}{2}\rfloor-1, \lceil\frac{d}{2}\rceil-1)}(F),\rk\Hess_\ell^{(\lfloor\frac{d-1}{2}\rfloor, \lfloor\frac{d}{2}\rfloor)}(F)\right).$$ 
%$$r=\rk\Hess_\ell^{(\lfloor\frac{d}{2}\rfloor-1, \lceil\frac{d}{2}\rceil-2)}(F), s=\rk\Hess_\ell^{(\lfloor\frac{d}{2}\rfloor-1, \lceil\frac{d}{2}\rceil-1)}(F), \hspace*{2mm} \text{and} \hspace*{2mm} t=\rk\Hess_\ell^{(\lfloor\frac{d-1}{2}\rfloor, \lfloor\frac{d}{2}\rfloor)}(F).$$ 
\noindent Moreover, if $t\neq 3r$, then $P_{\ell,A}$ is uniquely determined by non-zero integers $(r,s,t)$. Otherwise, if $t=3r$, then $P_{\ell,A}$ is uniquely determined by non-zero integers \linebreak$(r,\rk\Hess_\ell^{(r,d-r-1)}(F)).$

\item If $\ell^3=0$ and $\ell^2\neq 0$, then 
\begin{equation}\label{JT3}
P_{\ell,A}=(\underbrace{3,\dots ,3}_{\Delta^2\mathbf{d}(2)},\underbrace{2,\dots ,2}_{\Delta^2\mathbf{d}(1)}, \underbrace{1,\dots ,1}_{\Delta^2\mathbf{d}(0)}),
\end{equation}
where $\mathbf{d}=(\dim_\mathsf{k} A,\dim_\mathsf{k} A^{(1)},\dim_\mathsf{k} A^{(2)})$ and the Hilbert functions of $A^{(1)}$ and $A^{(2)}$ are given in Corollary \ref{2linescorollary} for parameters  
 $$(r,s)=\left(\rk\Hess_\ell^{(\lfloor\frac{d-1}{2}\rfloor, \lfloor\frac{d}{2}\rfloor)}(F),\rk\Hess_\ell^{(\lfloor\frac{d-1}{2}\rfloor, \lfloor\frac{d}{2}\rfloor-1)}(F)\right).$$
%$$r=\rk\Hess_\ell^{(\lfloor\frac{d-1}{2}\rfloor, \lfloor\frac{d}{2}\rfloor)}(F),\hspace*{2mm} \text{and} \hspace*{2mm}  s=\rk\Hess_\ell^{(\lfloor\frac{d-1}{2}\rfloor, \lfloor\frac{d}{2}\rfloor-1)}(F).$$
 Moreover, $P_{\ell,A}$ is uniquely determined by non-zero integers $(r,s).$

\item If $\ell^2=0$ and $\ell\neq 0$, then 
\begin{equation}\label{JT2}
P_{\ell,A}=(\underbrace{2,\dots ,2}_{\Delta^2\mathbf{d}(1)},\underbrace{1,\dots ,1}_{\Delta^2\mathbf{d}(0)}),
\end{equation}
where $\mathbf{d}=(\dim_\mathsf{k} A,\dim_\mathsf{k} A^{(1)})$ 
where the Hilbert function of $A^{(1)}$ is given in Corollary \ref{2linescorollary} for parameter $$r=\rk\Hess_\ell^{(\lfloor\frac{d-1}{2}\rfloor, \lfloor\frac{d}{2}\rfloor)}(F).$$
Moreover, $P_{\ell,A}$ is uniquely determined by the non-zero integer $r.$
\end{itemize}
\end{theorem}
\begin{proof}
First assume that $\ell^3\neq 0$ and notice that the socle degree of  $A^{(1)}=S/\ann(\ell\circ F)$ equals to $d-1$. Recall from Remark \ref{r_ij-remark} that
$$r=\rk\Hess_\ell^{(\lfloor\frac{d}{2}\rfloor-1,\lceil\frac{d}{2}\rceil-2 )}(F)= h_{A^{(3)}}(\lfloor\frac{d}{2}\rfloor-1), $$
$$s=\rk\Hess_\ell^{(\lfloor\frac{d}{2}\rfloor-1,\lceil\frac{d}{2}\rceil-1 )}(F)= h_{A^{(2)}}(\lfloor\frac{d}{2}\rfloor-1), $$
and 
$$t=\rk\Hess_\ell^{(\lfloor\frac{d-1}{2}\rfloor,\lfloor\frac{d}{2}\rfloor )}(F)= h_{A^{(1)}}(\lfloor\frac{d-1}{2}\rfloor).$$
Then using Theorems  \ref{3linesHFtheorem-even} and \ref{3linesHFtheorem-odd}, we get the ranks of multiplication maps by $\ell$,  $\ell^2$ and $\ell^3$ on $A$ in various degrees from the rank matrix of $A^{(1)}$ in terms of $r,s,t$. Then using Proposition \ref{JT}, we get $P_{\ell,A}$, as we claimed in  (\ref{JT4}). Moreover, we proved in Theorems  \ref{3linesHFtheorem-even} and \ref{3linesHFtheorem-odd} that the rank matrix of $A^{(1)}$ is uniquely determined in terms of $r,s$ and $t$ except when $t=3r$. In this case, there are two possible rank matrices for $A^{(1)}$ that is determined uniquely in terms of $r$ and $\Hess_\ell^{(r,d-r-1)}(F)$.

Now suppose that $\ell^3=0$ and $\ell^2\neq 0$. Ranks of multiplication maps on $A$ by $\ell$ and $\ell^2$ are uniquely determined by $r$ and $s$ in Corollary \ref{2linescorollary}, where 
$$r=\rk\Hess_\ell^{(\lfloor\frac{d-1}{2}\rfloor, \lfloor\frac{d}{2}\rfloor)}(F)=h_{A^{(2)}}(\lfloor\frac{d-1}{2}\rfloor),\hspace*{2mm} \text{and} \hspace*{2mm}  s=\rk\Hess_\ell^{(\lfloor\frac{d-1}{2}\rfloor, \lfloor\frac{d}{2}\rfloor-1)}(F)=h_{A^{(1)}}(\lfloor\frac{d-1}{2}\rfloor).$$ Therefore, Proposition \ref{JT} implies that $P_{\ell,A}$ is equal to (\ref{JT3}).

Assume that $\ell^2=0$ and $\ell\neq 0$ and that $r=\rk\Hess_\ell^{(\lfloor\frac{d-1}{2}\rfloor, \lfloor\frac{d}{2}\rfloor)}(F)=h_{A^{(1)}}(\lfloor\frac{d-1}{2}\rfloor)$.  Then, Corollary \ref{2linescorollary} provides the rank  of multiplication map by $\ell$, by providing  $h_{A^{(1)}}$ which implies the desired Jordan type in this case.
\end{proof}
\begin{remark}
One may use  \ref{3linesHFtheorem-even} and \ref{3linesHFtheorem-odd} to get the Jordan degree types with parts of length at most four, similar to the above theorem. Thus, such Jordan degree type is also determined uniquely by at most the ranks of three mixed Hessians.\par
More precise formulas for $P_{\ell,A}$ could be obtained directly from rank matrices provided in Theorems \ref{3linesHFtheorem-even} and \ref{3linesHFtheorem-odd}. 
\end{remark}
\begin{example}
Let $A=S/\ann(F)$ be an Artinian Gorenstein algebra where $$F=X^3Y^4+X^3Z^4+X^2YZ^4+Y^3Z^4.$$ We have that 
$$\rk\Hess^{(2,2)}_x(F)=2, \quad \rk\Hess^{(2,3)}_x(F)=4, \hspace*{2mm}\text{and}\hspace*{2mm}\rk\Hess^{(3,3)}_x(F)=7.$$
So using Theorem \ref{3linesHFtheorem-even} for $(r,s,t)=(2,4,7)$ we get the rank matrix for $A^{(1)}=S/\ann(x\circ F)$ and $\ell=x$ that is equal to
$$M_{x,A^{(1)}}= \begin{pmatrix}
1&1&1&0&0&0&0\\
0&3&3&2&0&0&0\\
0&0&6&4&2&0&0\\
0&0&0&7&4&2&0\\
0&0&0&0&6&3&1\\
0&0&0&0&0&3&1\\
0&0&0&0&0&0&1\\
\end{pmatrix}.
$$
Therefore, using  Equation \ref{JT4} we get the Jordan type of $A$ and $x$, which is equal to
$$
P_{x,A}=(\underbrace{4,\dots ,4}_8,\underbrace{2,\dots ,2}_3, \underbrace{1,\dots ,1}_{\dim_{\mathsf{k}}A-38})= (\underbrace{4,\dots ,4}_8,\underbrace{2,\dots ,2}_3)
$$
sicne we have $h_A=(1,3,6,9,9,6,3,1)$.
Moreover, using the correspondence to the Jordan degree type matrix in Proposition \ref{rkmatrix-1-1-JDT-prop} we get that the Jordan degree type partition of $A$ for $x$ is equal to $\mathcal{S}_{x,A}=(4_0,4_1,4_1,4_2,4_2,4_3,4_3,4_4,2_2, 2_3,2_4).$
\end{example}

Based on computations in Macaulay2, for large number of cases up to socle degree nine, we have no example of Artinian Gorenstein algebras over polynomial rings with three variables that the necessary conditions given in Lemmas \ref{diffO-seq} and \ref{additiveRank} are not sufficient. So we pose the following conjecture.

\begin{conjecture}
Let $M$ be an upper triangular matrix of size $d+1$ with non-negative entries. Then $M$ is the rank matrix of some Artinian Gorenstein algebra $A$ of codimension three and linear form $\ell\in A_1$, if and only if the following conditions are satisfied.
\begin{itemize}
\item[$(i)$] For every $0\leq i\leq d$, $\diag(i,M)$ is an O-sequences, and $h_A=\diag(0,M)$;
\item[$(ii)$] for every $0\leq i\leq d-1$, the difference vector $\diag(i,M)-\left(\diag(i+1,M)\right)_+$ is an O-sequences;
\item[$(iii)$] for any $2\times 2$ square submatrix of successive entries on and above the diagonal of $M$ of the form $\begin{pmatrix}
u&v\\
w&z\\
\end{pmatrix}$ we have that $w+v\geq u+z$.
\end{itemize}
\end{conjecture}

\section{Acknowledgment}
The author would like to thank Mats Boij for many helpful discussions and  Anthony Iarrobino for his comments on the first draft of this paper. Experiments using the algebra software Macaulay2 \cite{Mac2} were essential to get the ideas behind some of the proofs. This work was supported by the grant VR2013-4545. 
\bibliography{my.bib}{}
\bibliographystyle{plain}
\end{document}